\documentclass[a4paper,twoside,10pt]{amsart}
\usepackage{amsmath,amsfonts,amssymb,amsthm}
\usepackage{mathtools,subcaption}
\usepackage{color,graphicx}
\usepackage[a4paper]{geometry}
\usepackage[colorlinks=true,pdfstartview={XYZ null null 1.00},pdfview=FitH,citecolor=blue,urlcolor=customgreen]{hyperref}
\newtheorem{theorem}{Theorem}[section]
\newtheorem{lemma}[theorem]{Lemma}
\newtheorem{proposition}[theorem]{Proposition}
\theoremstyle{remark}
\newtheorem*{remark}{Remark}
\numberwithin{equation}{section}
\def\im{\mathrm{i}}
\def\d{\mathrm{d}}
\def\e{\mathrm{e}}
\def\O{\mathcal{O}}

\def\Re{\operatorname{Re}}

\def\arccot{\operatorname{arccot}}
\def\Im{\operatorname{Im}}
\def\erf{\operatorname{erf}}
\def\erfc{\operatorname{erfc}}
\definecolor{customgreen}{rgb}{0.0, 0.5, 0.0}
\author[G. Nemes]{Gerg\H{o} Nemes}

\address{Department of Analysis, Alfr\'ed R\'enyi Institute of Mathematics, Re\'altanoda utca 13--15, Budapest H-1053, Hungary}
\email{nemes.gergo@renyi.hu}

\address{Department of Physics, Tokyo Metropolitan University, 1--1 Minami-osawa, Hachioji-shi, Tokyo 192-0397, Japan}
\email{nemes@tmu.ac.jp}

\dedicatory{Dedicated to Sir Michael V. Berry on the occasion of his 80$^{\text{th}}$ birthday.}


\begin{document}

\title[Dingle's final main rule, Berry's transition, and Howls' conjecture]{Dingle's final main rule, Berry's transition,\\ and Howls' conjecture}

\begin{abstract} The Stokes phenomenon is the apparent discontinuous change in the form of the asymptotic expansion of a function across certain rays in the complex plane, known as Stokes lines, as additional expansions, pre-factored by exponentially small terms, appear in its representation. It was first observed by G.~G.~Stokes while studying the asymptotic behaviour of the Airy function. R.~B.~Dingle proposed a set of rules for locating Stokes lines and continuing asymptotic expansions across them. Included among these rules is the ``final main rule" stating that half the discontinuity in form occurs on reaching the Stokes line, and half on leaving it the other side. M.~V.~Berry demonstrated that, if an asymptotic expansion is terminated just before its numerically least term, the transition between two different asymptotic forms across a Stokes line is effected smoothly and not discontinuously as in the conventional interpretation of the Stokes phenomenon. On a Stokes line, in accordance with Dingle's final main rule, Berry's law predicts a multiplier of $\frac{1}{2}$ for the emerging small exponentials. In this paper, we consider two closely related asymptotic expansions in which the multipliers of exponentially small contributions may no longer obey Dingle's rule: their values can differ from $\frac{1}{2}$ on a Stokes line and can be non-zero only on the line itself. This unusual behaviour of the multipliers is a result of a sequence of higher-order Stokes phenomena. We show that these phenomena are rapid but smooth transitions in the remainder terms of a series of optimally truncated hyperasymptotic re-expansions. To this end, we verify a conjecture due to C.~J.~Howls.
\end{abstract}
\maketitle\vspace{-4.8pt}

\section{Introduction}

During the course of a Stokes phenomenon, an asymptotic expansion of a function can change its form near certain rays in the complex plane, known as Stokes lines, through the seemingly discontinuous appearance of further series with exponentially small pre-factors. It was first observed by G.~G.~Stokes \cite{Stokes1864} while studying the asymptotic behaviour of the Airy function. Although negligible at the place of their birth, the exponentially small contributions introduced by the Stokes phenomenon can grow to significantly influence or dominate the behaviour of the function in other regions. As noted in \cite{Howls2004}, due to the omission of exponentially small terms in the subsequently adopted definition of an asymptotic expansion formulated by H.~Poincar\'e \cite{Poincare1886}, the Stokes phenomenon caused controversy, ambiguity and misunderstanding for over a century.

During his work in solid-state physics, R.~B.~Dingle encountered mathematical difficulties, associated with the evaluation of some integrals, leading to divergent asymptotic series. Dingle realised the limitations of Poincar\'e's prescription and that existing techniques for interpreting asymptotic series were often vague and severely limited in accuracy and range of applicability. Building on previous works by T.~J.~Stieltjes, J.~R.~Airey and others, Dingle developed a new and extensive formal, interpretative theory of asymptotic series, and summarised his findings in his exhaustive and defining 1973 book, \emph{Asymptotic Expansions: Their Derivation and Interpretation} \cite{Dingle1973}. In his book, Dingle proposed a set of rules for locating Stokes lines and continuing asymptotic expansions across them. To illustrate these rules in the simplest case of a compound asymptotic expansion consisting of just two asymptotic series, we consider the parabolic cylinder functions. Let $U(a,z)$ and $V(a,z)$ denote those solutions of the parabolic cylinder equation
\begin{equation}\label{eq1}
\frac{\d^2 w}{\d z^2 } = \left( \tfrac{1}{4}z^2 + a \right)w, \quad a\in \mathbb{R},
\end{equation}
which are respectively exponentially decreasing and increasing when $z$ is real and positive, and are real-valued for real $z$. The solution $U(a,z)$, up to scalar multiplication, is uniquely determined by these properties, whereas $V(a,z)$ may be defined uniquely through power series \cite[\href{https://dlmf.nist.gov/12.4}{\S12.4}]{NIST:DLMF}. By trial substitution to the equation \eqref{eq1}, the asymptotic expansions of these solutions are found to be
\begin{equation}\label{eq2}
U(a,z) \sim \e^{ - \frac{1}{4}z^2 } z^{ - a - \frac{1}{2}} \sum\limits_{n = 0}^\infty ( - 1)^n \frac{\left( \frac{1}{2} + a \right)_{2n} }{n!(2z^2)^n} 
\end{equation}
and
\begin{equation}\label{eq3}
V(a,z) \sim \sqrt {\frac{2}{\pi}} \e^{\frac{1}{4}z^2} z^{a - \frac{1}{2}} \sum\limits_{n = 0}^\infty \frac{\left( \frac{1}{2} - a \right)_{2n} }{n!(2z^2 )^n} ,
\end{equation}
for large positive $z$ and fixed $a$, with $(w)_n=\Gamma(w+n)/\Gamma(w)$. The factor $\sqrt{2/\pi}$ in the expansion of $V(a,z)$ is inserted to match the normalisation convention in the literature, e.g., \cite[\href{https://dlmf.nist.gov/12}{Ch. 12}]{NIST:DLMF}. The rays $\arg z = \pm \frac{\pi}{2}$ are Stokes lines for $U(a,z)$ where the first series \eqref{eq2} maximally dominates the second series \eqref{eq3} (Dingle's second rule). This also follows from the like signs and phases of the late terms in \eqref{eq2} (Dingle's first rule). When the ray $\arg z = \frac{\pi}{2}$ is crossed the Stokes phenomenon switches on a divergent series proportional to \eqref{eq3} (Dingle's fourth rule). When $\arg z = \pi$, the second series \eqref{eq3} is at peak dominance relative to the first series, so has a Stokes line there (Dingle's second rule). This is also clear from the like signs and phases of the late terms in \eqref{eq3} (Dingle's first rule). Consequently, the asymptotic behaviour of $U(a,z)$ in the sector $|\arg z|<\frac{\pi}{2}$ is described by the expansion \eqref{eq2}, while in the sector $\frac{\pi}{2} < \arg z < \pi$ we have the compound expansion
\begin{equation}\label{eq4}
U(a,z) \sim \e^{ - \frac{1}{4}z^2 } z^{ - a - \frac{1}{2}} \sum\limits_{n = 0}^\infty ( - 1)^n \frac{\left( \frac{1}{2} + a \right)_{2n} }{n!(2z^2)^n} + S \sqrt {\frac{2}{\pi}} \e^{\frac{1}{4}z^2} z^{a - \frac{1}{2}} \sum\limits_{n = 0}^\infty \frac{\left( \frac{1}{2} - a \right)_{2n} }{n!(2z^2 )^n}.
\end{equation}
The magnitude of the constant $S$, which may depend on $a$ but is independent of $z$, is unknown at this stage. Its value can be determined by continuing the analysis for the next phase sector $\pi < \arg z < \frac{3\pi}{2}$ and alluding to the fact that $U(a,z)$ is real when $z$ is real (cf. \cite[pp. 9--10]{Dingle1973}), or by examining the asymptotic behaviour of the high-order coefficients in \eqref{eq2} \cite{ABOD1994} (see also \cite[\S16.511]{WW1927} for yet another derivation using power series). With either of these methods, it is found that
\[
S  = \frac{\pi \im}{\Gamma \left( \frac{1}{2} + a \right)}\e^{ - \pi \im a}.
\]
If $S$ is regarded not as a constant but as a function of $\theta=\arg z$, the expansion \eqref{eq4} applies in the larger sector $-\frac{\pi}{2} < \arg z < \pi$ provided we set
\begin{equation}\label{eq5}
S=S(\theta)=\begin{cases} 0 & \text{if }\; |\theta|<\frac{\pi}{2}, \\[0.25em] \dfrac{\pi \im}{2\Gamma \left( \frac{1}{2} + a \right)}\e^{ - \pi \im a} & \text{if }\; \theta = \frac{\pi}{2}, \\[0.8em] \dfrac{\pi \im}{\Gamma \left( \frac{1}{2} + a \right)}\e^{ - \pi \im a} & \text{if }\; \frac{\pi}{2} < \theta < \pi. \end{cases}
\end{equation}
The $S$ is called the Stokes multiplier. The value of $S$ on the Stokes line $\theta = \frac{\pi}{2}$ was inferred from Dingle's ``final main rule", namely that half the discontinuity in form occurs on reaching the Stokes line, and half on leaving it the other side. In Dingle's terminology, the combination of \eqref{eq4} and \eqref{eq5} constitutes the complete asymptotic expansion of $U(a,z)$ in the sector $-\frac{\pi}{2}<\arg z <\pi$. A similar analysis may be carried out for the conjugate sector $-\pi<\arg z <-\frac{\pi}{2}$ and for the expansion \eqref{eq3} of $V(a,z)$.

This is the conventional view of the Stokes phenomenon, where the multiplier of the exponentially small contribution changes abruptly on crossing a Stokes line. In his insightful paper \cite{Berry1989}, M.~V.~Berry adopted the view that the Stokes multiplier should be regarded as a continuous function of $\arg z$, rather than a discontinuous one. In our example above, this is achieved by terminating the two series in \eqref{eq4}: the first series is truncated optimally (that is, just before its least term) after $N$ terms, where $N= \lfloor \frac{1}{2}\left| z \right|^2 -a \rfloor$, and the second series is truncated at its first term. A modified Stokes multiplier $\mathcal{S}_N$ for the Stokes line $\arg z =\frac{\pi}{2}$ is then defined exactly by the identity
\[
U(a,z) = \e^{ - \frac{1}{4}z^2 } z^{ - a - \frac{1}{2}} \sum\limits_{n = 0}^{N-1} ( - 1)^n \frac{\left( \frac{1}{2} + a \right)_{2n} }{n!(2z^2)^n} +  \mathcal{S}_N \frac{\sqrt{2\pi}}{\Gamma \left( \frac{1}{2} + a\right)}\e^{\frac{1}{4}z^2 } (z\e^{ - \pi \im} )^{a - \frac{1}{2}} .
\]
The $\mathcal{S}_N=\mathcal{S}_N(\theta)$ is a piecewise continuous function of $|z|$, and a continuous function of $\theta=\arg z$. Berry demonstrated that if $|z|$ is large and held fixed, then $\mathcal{S}_N$ changes rapidly but smoothly from approximately $0$ to approximately $1$ as $\theta$ increases continuously from values somewhat below $\frac{\pi}{2}$ to values somewhat above $\frac{\pi}{2}$. This is consistent with \eqref{eq4} and \eqref{eq5} above. In other words, Berry showed that the contribution of the exponentially small term in the neighbourhood of the Stokes line $\arg z =\frac{\pi}{2}$ is introduced smoothly, and not as an abrupt jump as was commonly supposed. Moreover, Berry demonstrated that for a wide class of asymptotic expansions, this smooth transition of the Stokes multiplier obeys a universal form given approximately by an error function whose argument is an appropriate variable describing the transition across a Stokes line. In our example, Berry's approximation takes the form
\[
\mathcal{S}_N \approx \frac{1}{2} + \frac{1}{2}\erf\left(\frac{1}{2}\frac{\Im(-z^2)}{\sqrt{\Re(-z^2)}}\right) = \frac{1}{2}\erfc\left(\frac{1}{2}\frac{\Im (z^2)}{\sqrt{\Re(-z^2)}}\right),
\]
where $\erf$ and $\erfc$ are the error function and the complementary error function, respectively. Thus, $\mathcal{S}_N$ increases smoothly from approximately $0$ when $\Im(z^2)/\sqrt{\Re(-z^2)}$ is large and negative (i.e., $|z|$ is large and $\arg z$ is just below $\frac{\pi}{2}$) to approximately $1$ when $\Im(z^2)/\sqrt{\Re(-z^2)}$ is large and positive (i.e., $|z|$ is large and $\arg z$ is just above $\frac{\pi}{2}$). On the Stokes line $\arg z =\frac{\pi}{2}$, $\Im(z^2)/\sqrt{\Re(-z^2)}=0$ and therefore $\mathcal{S}_N \approx \frac{1}{2}$, in accordance with Dingle's final main rule. 

Berry's approximation for the Stokes multiplier $\mathcal{S}_N$ is derived by application of Dingle's theory of terminants. Dingle observed that, although the early terms of an asymptotic series can rapidly get extremely complicated, the high-order coefficients of a wide class of asymptotic expansions display a universal and simple form. (This phenomenon may be seen as a consequence of G.~M.~Darboux's approximation of high derivatives in terms of factorials \cite{Darboux1878}.) By a clever application of \'E.~Borel's summation method to the asymptotic form of the high orders, Dingle transformed the divergent tail of an asymptotic expansion into a new series in terms of certain integrals that he called ``basic terminants". The resulting series, Dingle's ``interpretation" of the asymptotic expansion, can yield exceedingly accurate approximations and automatically incorporates the Stokes phenomenon. He envisaged in his book that ``these terminant expansions can themselves be closed with new terminants; and so on stage after stage". Such re-summations are now known as hyperasymptotics and have a well-established theory \cite{BHNOD2018,BerryHowls1990,BerryHowls1991,ABOD1998}.

In this paper, we revisit two familiar, closely related asymptotic expansions, namely those of the gamma function $\Gamma(z)$ and its reciprocal. These asymptotic expansions exhibit a significantly different behaviour than that discussed above: in the neighbourhood of the rays $\arg z =\pm\frac{\pi}{2}$, not one but infinitely many exponentially small contributions appear, each associated with its own Stokes multiplier. Furthermore, these contributions may no longer obey Dingle's rule: the value of their associated Stokes multipliers can differ from $\frac{1}{2}$ on the Stokes line and may be non-zero only on the line itself. This unusual behaviour of the multipliers is a result of a sequence of higher-order Stokes phenomena. We will demonstrate that if the asymptotic expansions are interpreted as a series of optimally truncated hyperasymptotic re-expansions, the multiplier of each subdominant exponential experiences a smooth, but rapid, change in value described by a multivariate polynomial in complementary error functions. This result is a natural extension of Berry's theory to a particular, yet important, case of the higher-order Stokes phenomenon. Our proof relies crucially on a non-trivial identity between hyperterminants, which was originally conjectured by C.~J.~Howls.

The remaining part of the paper is organised as follows. In Section \ref{section2}, we derive the complete asymptotic expansions of the gamma function and its reciprocal and determine the Stokes multipliers in the discontinuous sense. The smooth interpretation of the (higher-order) Stokes phenomena via hyperasymptotics is described in Section \ref{section3}. In Section \ref{section4}, we settle Howls' conjecture on the hyperterminants. Section \ref{section5} contains the main result of the paper: the uniform asymptotic approximations for the hyperterminants in terms of complementary error functions. The paper concludes with a short discussion in Section \ref{section6}.

\section{The Stokes phenomenon associated with the gamma function and~its~reciprocal}\label{section2}

A well-known representation of the logarithm of the gamma function is given by
\begin{equation}\label{eq7}
\log \Gamma (z) = \left(z - \frac{1}{2}\right)\log z - z + \frac{1}{2}\log (2\pi ) + g(z),
\end{equation}
where $g(z)$ is an analytic function of $z$ on the principal sheet $|\arg z|<\pi$ of the complex logarithm. Several explicit representations for $g(z)$ are available in the literature (see, e.g., \cite[\S2.1.1]{Paris2001}). It is well known that as $z\to \infty$ in $|\arg z|<\frac{\pi}{2}$, $g(z)$ possesses the asymptotic expansion
\begin{equation}\label{eq6}
g(z) \sim \sum\limits_{n = 1}^\infty\frac{B_{2n}}{2n(2n-1)z^{2n-1}} =\frac{1}{12z}-\frac{1}{360z^3}+\frac{1}{1260z^5} - \ldots
\end{equation}
where $B_n$ are the Bernoulli numbers (see \cite[\href{https://dlmf.nist.gov/5.11.E1}{Eq. 5.11.1}]{NIST:DLMF}). Since successive even-order Bernoulli numbers have opposite signs, all terms in the asymptotic expansion \eqref{eq6} have the same phase when $\arg z = \pm \frac{\pi}{2}$, so these are Stokes lines for $\log \Gamma (z)$.

To obtain the expansion in the sectors $\frac{\pi}{2} < |\arg z| < \pi$, we can use the reflection formula in the form
\[
\Gamma ( - z) =  - \frac{\pi }{z\sin (\pi z)}\frac{1}{\Gamma (z)}.
\]
Substitution into \eqref{eq7} shows that $g(z)$ must satisfy the functional relation
\begin{equation}\label{eq8}
g(z) =  - g(z\e^{ \mp \pi \im} ) - \log (1 - \e^{ \pm 2\pi \im z} ) =  - g(z\e^{ \mp \pi \im} ) + \sum\limits_{k = 1}^\infty \frac{\e^{ \pm 2\pi \im k z}}{k} .
\end{equation}
The upper or lower sign is taken according as $z$ is in the upper or lower left half-plane. Thus the analytic continuation of $g(z)$ has generated an infinite string of exponential terms. These terms are subdominant compared to the series \eqref{eq6} in the upper and lower half-planes and are maximally subdominant on the Stokes lines $\arg z = \pm \frac{\pi}{2}$, respectively. As $z$ approaches the negative real axis, however, these exponentials steadily grow in magnitude and add to generate the poles of $\Gamma(z)$.

The combination of \eqref{eq7}, \eqref{eq6} and \eqref{eq8}, and Dingle's final main rule yields the complete asymptotic expansion of $\log \Gamma (z)$ in the sector $|\arg z|<\pi$:
\begin{equation}\label{eq11}
\log \Gamma (z) \sim \left(z - \frac{1}{2}\right)\log z - z + \frac{1}{2}\log (2\pi ) + \sum\limits_{n = 1}^\infty\frac{B_{2n}}{2n(2n-1)z^{2n-1}}+\sum\limits_{k = 1}^\infty S^{(k)} \frac{\e^{ \pm 2\pi \im k z}}{k},
\end{equation}
where the Stokes multipliers $S^{(k)}$ are
\[
S^{(k)}=S^{(k)}(\theta)=\begin{cases} 0 & \text{if }\; |\theta|<\frac{\pi}{2}, \\[0.25em] \frac{1}{2} & \text{if }\; \theta = \pm \frac{\pi}{2}, \\[0.25em] 1 & \text{if }\; \frac{\pi}{2} < |\theta| < \pi, \end{cases}
\]
with $\theta=\arg z$. R.~B.~Paris and A.~D.~Wood \cite{PW1992} gave an alternative derivation of the expansions on the Stokes lines by invoking the specific form of $\left|\Gamma (\im y)\right|$, $y>0$.

It is expected that the small exponentials emerge from the remainder of the optimally truncated divergent series \eqref{eq6} in a smooth manner, according to the error function law, as $z$ crosses the rays $\arg z = \pm \frac{\pi}{2}$. The demonstration of the smooth transition of the leading subdominant exponentials ($k = 1$ in \eqref{eq11}) was first given by Paris and Wood \cite{PW1992}. Berry \cite{Berry1991} proved by a different method, involving a sequence of increasingly delicate subtractions of optimally-truncated asymptotic series, that all the smaller exponentials appear in this way. His result was later reproduced by Paris and D.~Kaminski using Mellin--Barnes integrals \cite[\S6.4.2]{Paris2001} (see also \cite{Paris1993}).

Our primary interest in this paper is concerned with the asymptotics, and the related Stokes phenomena, of the gamma function itself rather than its logarithm. With regard to this case, it is convenient to introduce the scaled gamma function $\Gamma^\ast (z)$ defined by
\[
\Gamma^\ast (z) = \frac{\Gamma (z)}{\sqrt{2\pi } z^{z-1/2} \e^{-z}} = \e^{g(z)} .
\]
The function $\Gamma^\ast (z)$ can be analytically continued as a meromorphic function on the Riemann surface $\widehat{\mathbb C}$ associated with the complex logarithm. Exponentiation of \eqref{eq6} produces the familiar asymptotic expansion
\begin{equation}\label{eq9}
\Gamma^\ast (z) \sim \sum\limits_{n = 0}^\infty(-1)^n\frac{\gamma_n}{z^n} = 1 + \frac{1}{12z} + \frac{1}{288z^2} - \frac{139}{51840z^3} - \ldots,
\end{equation}
as $z\to \infty$ in the sector $|\arg z|<\frac{\pi}{2}$, where $\gamma_n$ are the so-called Stirling coefficients (see, e.g., \cite[\S2.1.2]{Paris2001}). The corresponding expansion for the reciprocal function is given by
\begin{equation}\label{eq10}
\frac{1}{\Gamma^\ast(z)} \sim \sum\limits_{n = 0}^\infty \frac{\gamma_n}{z^n} = 1 - \frac{1}{12z} + \frac{1}{288z^2} + \frac{139}{51840z^3} - \ldots,
\end{equation}
as $z\to \infty$ in $|\arg z|<\frac{\pi}{2}$. Observe that \eqref{eq10} involves the same set of coefficients as \eqref{eq9} but with different signs of the coefficients with odd index. This may be explained by noting that the series \eqref{eq6} has odd powers only \cite[p. 63]{Temme1996}. We remark that the reciprocal function is analytic on $\widehat{\mathbb C}$.

The Stokes phenomena associated with the gamma function and its reciprocal are directly inherited from that of the function $\log \Gamma (z)$. In particular, the positive and negative imaginary axes are Stokes lines. There are however significant differences which only become apparent when studying complete asymptotic expansions and the values of Stokes multipliers on the rays $\arg z = \pm \frac{\pi}{2}$. Indeed, exponentiating \eqref{eq11} and using the standard Maclaurin series of the functions $\exp ( \mp \log (1 - w)) = (1 - w)^{ \mp 1}$ and $\exp\big( \mp \frac{1}{2}\log (1 - w)\big) = (1 - w)^{ \mp 1/2}$ yields the complete asymptotic expansions of $\Gamma^\ast (z)$ and its reciprocal in the sector $|\arg z|<\pi$:
\begin{equation}\label{eq12}
\Gamma^\ast (z) \sim \sum\limits_{n = 0}^\infty (-1)^n\frac{\gamma_n}{z^n} +\sum\limits_{k=1}^\infty\mathbf{S}^{(k)} \e^{ \pm 2\pi \im kz} \sum\limits_{n = 0}^\infty (-1)^n \frac{\gamma _n }{z^n}
\end{equation}
and
\begin{equation}\label{eq13}
\frac{1}{\Gamma^\ast(z)} \sim \sum\limits_{n = 0}^\infty \frac{\gamma_n}{z^n} - \sum\limits_{k = 1}^\infty \widetilde{\mathbf{S}}^{(k)} \e^{ \pm 2\pi \im kz} \sum\limits_{n = 0}^\infty \frac{\gamma_n}{z^n},
\end{equation}
where the Stokes multipliers $\mathbf{S}^{(k)}$ and $\widetilde{\mathbf{S}}^{(k)}$ are
\begin{equation}\label{eq23}
\mathbf{S}^{(k)}=\mathbf{S}^{(k)}(\theta)=\begin{cases} 0 & \text{if }\; |\theta|<\frac{\pi}{2}, \\[0.25em] \frac{1}{k!}\left(\frac{1}{2}\right)_k & \text{if }\; \theta = \pm \frac{\pi}{2}, \\[0.25em] 1 & \text{if }\; \frac{\pi}{2} < |\theta| < \pi, \end{cases}
\end{equation}
for $k\geq 1$, and
\begin{equation}\label{eq24}
\widetilde{\mathbf{S}}^{(1)}=\widetilde{\mathbf{S}}^{(1)}(\theta)=\begin{cases} 0 & \text{if }\; |\theta|<\frac{\pi}{2}, \\[0.25em] \frac{1}{2} & \text{if }\; \theta = \pm \frac{\pi}{2}, \\[0.25em] 1 & \text{if }\; \frac{\pi}{2} < |\theta| < \pi, \end{cases}
\quad
\widetilde{\mathbf{S}}^{(k)}=\widetilde{\mathbf{S}}^{(k)}(\theta)=\begin{cases} 0 & \text{if }\; |\theta|<\frac{\pi}{2}, \\[0.25em] \frac{-1}{k!}\left(-\frac{1}{2}\right)_k & \text{if }\; \theta = \pm \frac{\pi}{2}, \\[0.25em] 0 & \text{if }\; \frac{\pi}{2} < |\theta| < \pi, \end{cases}
\end{equation}
for $k\geq 2$, and with $\theta=\arg z$ (cf. \cite[\S3.2]{Nemes2015}). The upper or lower signs are taken in \eqref{eq12} and \eqref{eq13} according as $z$ is in the upper or lower left half-plane. Expansions of type \eqref{eq12} and \eqref{eq13} are commonly referred to as transseries expansions. It is seen from \eqref{eq12} that when either the positive or the negative imaginary axis is crossed the Stokes phenomenon switches on an infinite copy of the expansion \eqref{eq9}, pre-factored by increasingly subdominant exponentially small terms. But the situation is different from the case of $\log\Gamma(z)$, in that the Stokes multipliers of the sub-subdominant contributions (corresponding to $k\geq 2$) do not follow Dingle's final main rule. The complete asymptotic expansion \eqref{eq13} of the reciprocal function exhibits an even more striking behaviour: an infinite number of exponentially small contributions of decreasing magnitude appear on each of the Stokes lines $\arg z = \pm \frac{\pi}{2}$ from which only the leading one ($k=1$ in \eqref{eq13}) survives on the respective other side. (The disappearance of the exponentials
is essential since the reciprocal gamma function is analytic along the negative real axis.) This unconventional behaviour of the Stokes multipliers is a consequence of a series of higher-order Stokes phenomena taking place across the lines $\arg z = \pm \frac{\pi}{2}$. These phenomena are not present in the expansion \eqref{eq6} of $\log\Gamma(z)$. At a higher-order Stokes phenomenon, a Stokes multiplier itself can change value, and in particular the potential for a Stokes phenomenon to occur is changed (for a general account see \cite{Howls2004}). This is indeed reflected in the expressions for $\mathbf{S}^{(k)}$ and $\widetilde{\mathbf{S}}^{(k)}$, $k\geq 2$.

The smooth interpretation of the Stokes phenomenon associated with the leading subdominant exponentials $\e^{ \pm 2\pi \im z}$ in \eqref{eq12} and \eqref{eq13} was studied by W.~G.~C.~Boyd \cite{Boyd1994} and the present author \cite{Nemes2015}. They verified that these small exponentials are born from the remainders of the optimally truncated expansions \eqref{eq9} and \eqref{eq10}, and obey Berry's error function smoothing law. As Berry \cite{Berry1991} remarks, the smoothings of the smaller exponentials which are apparent with $\log\Gamma(z)$ are obscured if one considers $\Gamma(z)$ instead, and their recovery might require hyperasymptotics. In the subsequent section we shall show this indeed to be the case: the smaller exponentials are born from the remainder terms of successive hyperasymptotic re-expansions. However, there are two major differences compared with $\log \Gamma(z)$. First, the switching-on of successively smaller exponential contributions can only be revealed one at a time, and second, the appearance of the sub-subdominant expansions is no longer described by a single error function. (We remark that an example of the non-universality of the error function smoothing law was also found by S.~J.~Chapman \cite{Chapman1996} in the solution of an inhomogeneous delay equation.)

\section{Hyperasymptotics and Stokes smoothings}\label{section3}

In this section, we discuss the hyperasymptotics process and the smoothings of the (higher-order) Stokes phenomena for the gamma function and its reciprocal. Briefly, hyperasymptotics is the systematic re-expansion of successive remainder terms in an asymptotic expansion that automatically incorporates the Stokes phenomenon. Each step of the process reduces the size of the error term by an exponentially small factor.

For any positive integer $N$, we denote by $R_N (z)$ and $\widetilde R_N (z)$ the remainders of the series \eqref{eq9} and \eqref{eq10} truncated to $N$ terms:
\begin{equation}\label{eq16}
\Gamma^\ast (z) = \sum\limits_{n = 0}^{N - 1} ( - 1)^n \frac{\gamma _n }{z^n }  + R_N (z)\quad \text{ and } \quad \frac{1}{\Gamma^\ast (z)} = \sum\limits_{n = 0}^{N - 1} \frac{\gamma _n }{z^n }  + \widetilde R_N (z).
\end{equation}
Note that for each $N$, $R_N (z)$ and $\widetilde R_N (z)$ are analytic functions of $z$ in the sector $|\arg z|<\pi$. The basis of hyperasymptotics will be the exact resurgence formulae (\cite[p. 616]{Boyd1994}, \cite[p. 576]{Nemes2015})
\begin{equation}\label{eq14}
R_N (z) = \frac{1}{2\pi \im}\frac{1}{z^N}\int_0^{[\pi/2^-]} \frac{\e^{2\pi \im t} t^{N - 1} }{1 - t/z}\Gamma^\ast (t)\d t  - \frac{1}{2\pi \im}\frac{1}{z^N}\int_0^{[-\pi/2^+]} \frac{\e^{ - 2\pi \im t} t^{N - 1}  }{1 - t/z}\Gamma^\ast (t)\d t 
\end{equation}
and
\begin{equation}\label{eq15}
\widetilde R_N (z) = - \frac{1}{2\pi \im}\frac{1}{z^N}\int_0^{[\pi /2^-]} \frac{\e^{2\pi \im t} t^{N - 1} }{1 - t/z}\Gamma^\ast(t\e^{-\pi \im} )\d t  + \frac{1}{2\pi \im}\frac{1}{z^N}\int_0^{[ - \pi /2^ +  ]} \frac{\e^{ - 2\pi \im t} t^{N - 1} }{1 - t/z}\Gamma^\ast (t\e^{\pi \im} )\d t .
\end{equation}
Here and throughout the paper we use the notation $\int_0^{\left[ \eta  \right]} = \int_0^{ + \infty \e^{\im\eta } }$. These formulae are valid under the assumption that $|\arg z|<\frac{\pi}{2}$, and they both incorporate the Stokes phenomenon. For example as $\arg z$ increases beyond $\frac{\pi}{2}$, the pole at $t=|z|\e^{\pi \im/2}$ in the first integral of \eqref{eq15} is entrapped: consequently for $\frac{\pi}{2}<\arg z<\frac{3\pi}{2}$, the term $- \e^{2\pi \im z} \Gamma^\ast (z\e^{ - \pi \im} )$ has to be added to the right-hand side of \eqref{eq15}. Expanding $\Gamma^\ast (z\e^{ - \pi \im} )$ via \eqref{eq10} then reproduces \eqref{eq13}.

We now define certain multiple integrals, the so-called ``hyperterminants", which form the basis of the hyperasymptotics scheme. These hyperterminants are multidimensional generalisations of Dingle's basic terminants. For $m$ a non-negative integer, we define
\begin{equation}\label{eq31}
F^{(0)}(z)=1,\qquad F^{(1)} \!\left( z;\! \begin{array}{c} N_1\!\! \\ \sigma_1\!\! \\ \end{array} \right) = \int_0^{[\pi-\arg \sigma_1]} \frac{\e^{\sigma_1 t_1} t_1^{N_1-1}}{z - t_1}\d t_1,
\end{equation}
\begin{equation}\label{eq34}
F^{(m)}\!\left( z;\! \begin{array}{c}
   N_1 ,\,\ldots,\!\!\!\!  \\ 
   \,\sigma _1 ,\,\,\ldots,\!\!\!\!  \\
\end{array}\begin{array}{c}
   N_m \!\!  \\
   \sigma _m \!\!  \\
\end{array} \right) = \int_0^{[\pi - \arg \sigma _1]} \frac{\e^{\sigma _1 t_1 } t_1^{N_1  - 1} }{z - t_1 } \prod_{k=2}^m \int_0^{ [\pi-\arg \sigma _k]} \frac{\e^{\sigma _k t_k } t_k^{N_k  - 1} }{t_{k - 1}  - t_k }\d t_k \d t_1,
\end{equation}
for an arbitrary set of complex numbers $N_1,\ldots,N_m$ such that $\Re(N_k)>1$ with $k=1,\ldots,m$, and for an arbitrary set of elements $\sigma_1,\ldots,\sigma_m$ of $\widehat{\mathbb C}$ such that $\arg \sigma_k \neq \arg \sigma_{k+1} \pmod{2\pi}$ with $k=1,\ldots,m$. These multiple integrals converge provided $|\arg(\sigma_1 z)|<\pi$, and can be extended by analytic continuation to other values of $\arg z$. The $F^{(m)}$ is termed an $m$th-level hyperterminant. We will refer to the quantities $\sigma_k$ as ``singulants", a term originating from the work of Dingle. We remark that
\begin{equation}\label{eq33}
F^{(1)} \!\left( z;\! \begin{array}{c} N\! \\ \sigma\! \\ \end{array} \right) =\e^{\pi \im N} \e^{\sigma z} z^{N - 1} \Gamma (N)\Gamma (1 - N,\sigma z),
\end{equation}
where $\Gamma(a,w)$ is the incomplete gamma function. In the case that $\arg \sigma_k = \arg \sigma_{k+1} \pmod{2\pi}$ for some $k$, we have to specify whether the $t_k$-contour is indented to the left or right of the pole of $(t_{k-1}-t_k)^{-1}$. We make the choice by defining
\begin{equation}\label{eq27}
F^{(m)}\!\left( z;\! \begin{array}{c}
   N_1 ,\,\ldots,\!\!\!\!  \\ 
   \,\sigma _1 ,\,\,\ldots,\!\!\!\!  \\
\end{array}\begin{array}{c}
   N_m \!\!  \\
   \sigma _m \!\!  \\
\end{array} \right)  = \mathop {\lim }\limits_{\varepsilon  \to 0^+ } F^{(m)} \!\left( z; \!\begin{array}{c}
   N_1 ,\!\!\!\!  \\
   \sigma _1 \e^{(m - 1)\varepsilon \im} ,\!\!\!\!  \\
\end{array}\begin{array}{c}
   N_2 , \quad\quad\quad\,\ldots, \!\!\!\!  \\
   \sigma _2 \e^{(m - 2)\varepsilon \im} ,\, \ldots, \!\!\!\!  \\
\end{array}\begin{array}{c}
   N_m \!\!  \\
   \sigma _m \!\!  \\
\end{array} \right).
\end{equation}
Thus, we require that the pole of $(t_{k-1}-t_k)^{-1}$ is on the left-hand side of the $t_k$-contour.

Throughout the paper, we will frequently use the connection formula
\begin{gather}\label{eq49}
\begin{split}
F^{(m)} \!\left( z\e^{-2\pi \im};\! \begin{array}{c}
   N_1 ,\,\ldots,\!\!\!\!  \\ 
   \,\sigma _1 ,\,\,\ldots,\!\!\!\!  \\
\end{array}\begin{array}{c}
   N_m \!\!  \\
   \sigma _m \!\!  \\
\end{array} \right) = \; & F^{(m)} \!\left( z;\! \begin{array}{c}
   N_1 ,\,\ldots,\!\!\!\!  \\ 
   \,\sigma _1 ,\,\,\ldots,\!\!\!\!  \\
\end{array}\begin{array}{c}
   N_m \!\!  \\
   \sigma _m \!\!  \\
\end{array} \right) \\& - 2\pi \im \e^{\sigma_1 z} z^{N_1-1} F^{(m-1)} \!\left( z;\! \begin{array}{c}
   N_2 ,\,\ldots,\!\!\!\!  \\ 
   \,\sigma _2 ,\,\,\ldots,\!\!\!\!  \\
\end{array}\begin{array}{c}
   N_m \!\!  \\
   \sigma _m \!\!  \\
\end{array} \right),
\end{split}
\end{gather}
valid when $m\geq 1$ \cite[Eq. (2.6)]{ABOD1997}.

The numerical evaluation of the hyperterminants is discussed in the paper by T.~B.~Bennett et al. \cite[Theorem A.1]{BHNOD2018} (see \cite{ABOD1997} for basic properties).

\subsection{Level 1 hyperasymptotics} We now derive the Level 1 hyperasymptotic expansions. In the integral representations \eqref{eq14} and \eqref{eq15} for the remainders $R_N(z)$ and $\widetilde R_N (z)$ we substitute \eqref{eq16} (with $M$ in place of $N$) into $\Gamma^\ast (t)$ and $\Gamma^\ast(t\e^{\mp\pi \im} )$. We obtain, after term-by-term integration, the re-expansions
\begin{gather}\label{eq17}
\begin{split}
R_N (z) = \; & \frac{1}{2\pi \im}\frac{1}{z^{N - 1} }\sum\limits_{m = 0}^{M - 1} ( - 1)^m \gamma _m F^{(1)} \!\left( z;\!\begin{array}{c}
   N - m\! \\
   2\pi \e^{\frac{\pi }{2}\im}\! \\
\end{array} \right)
\\ & - \frac{1}{2\pi \im}\frac{1}{z^{N - 1}}\sum\limits_{m = 0}^{M - 1} ( - 1)^m \gamma _m F^{(1)}\! \left( z;\!\begin{array}{c}
   N - m\! \\
   2\pi \e^{ - \frac{\pi }{2}\im}\! \\
\end{array} \right)  + R_{N,M} (z)
\end{split}
\end{gather}
and
\begin{gather}\label{eq18}
\begin{split}
\widetilde R_N (z) = & - \frac{1}{2\pi \im}\frac{1}{z^{N - 1} }\sum\limits_{m = 0}^{M - 1}\gamma _m F^{(1)} \!\left( z;\!\begin{array}{c}
   N - m \! \\
   2\pi \e^{\frac{\pi }{2}\im} \!  \\
\end{array} \right)
\\ & + \frac{1}{2\pi \im}\frac{1}{z^{N - 1}}\sum\limits_{m = 0}^{M - 1} \gamma _m F^{(1)}\! \left( z;\!\begin{array}{c}
   N - m \! \\
   2\pi \e^{ - \frac{\pi }{2}\im} \! \\
\end{array} \right)  + \widetilde R_{N,M} (z),
\end{split}
\end{gather}
with
\begin{equation}\label{eq25} 
R_{N,M} (z) = \frac{1}{2\pi \im}\frac{1}{z^N}\int_0^{[\pi/2^-]} \frac{\e^{2\pi \im t} t^{N - 1} }{1 - t/z}R_M(t)\d t  - \frac{1}{2\pi \im}\frac{1}{z^N}\int_0^{[-\pi/2^+]} \frac{\e^{ - 2\pi \im t} t^{N - 1}  }{1 - t/z}R_M(t)\d t
\end{equation}
and
\[ 
\widetilde R_{N,M} (z) = - \frac{1}{2\pi \im}\frac{1}{z^N}\int_0^{[\pi /2^-]} \frac{\e^{2\pi \im t} t^{N - 1} }{1 - t/z}R_M(t\e^{-\pi \im} )\d t  + \frac{1}{2\pi \im}\frac{1}{z^N}\int_0^{[ - \pi /2^ +  ]} \frac{\e^{ - 2\pi \im t} t^{N - 1} }{1 - t/z}R_M(t\e^{\pi \im} )\d t ,
\]
provided $M<N$. The remainders $R_{N,M} (z)$ and $\widetilde R_{N,M} (z)$ are defined initially for $|\arg z|<\frac{\pi}{2}$, but can be extended by a standard analytic continuation argument to $|\arg z|<\pi$.

For the sake of completeness we shall now reproduce the results of \cite{Boyd1994,Nemes2015}. We will show that the first few terms of the series in \eqref{eq12} and \eqref{eq13} with the exponential pre-factor $\e^{2\pi \im z}$ are switched on smoothly across the Stokes line $\arg z =\frac{\pi}{2}$ according to Berry's universal error function law. For this we examine the behaviour of the remainder terms $R_N (z)$ and $\widetilde R_N (z)$ in the vicinity of the ray $\arg z =\frac{\pi}{2}$. Assume that the asymptotic expansions \eqref{eq9} and \eqref{eq10} are truncated so that the optimal value for $N$, $N=\left\lfloor 2\pi |z|\right\rfloor$ is chosen in \eqref{eq16} (cf. \cite[p. 621]{Boyd1994}). Under this assumption it can be shown that for each fixed $M$,
\begin{equation}\label{eq22}
R_{N,M} (z),\; \widetilde R_{N,M} (z) =\begin{cases} \O(\left|z\right|^{-M}\e^{-2\pi\left|z\right|}) & \text{if }\; |\arg z|\leq\frac{\pi}{2}, \\[0.5em] \O(\left|z\right|^{-M}\e^{-2\pi\Im(z)}) & \text{if }\; \frac{\pi}{2} < \arg z \leq \pi-\delta < \pi, \end{cases}
\end{equation}
as $z\to \infty$, uniformly with respect to $\arg z$ \cite[Theorem 1.4]{Nemes2015}. The behaviour of $R_N (z)$ and $\widetilde R_N (z)$ as $z\to \infty$ in the neighbourhood of the line $\arg z =\frac{\pi}{2}$ can now be deduced from \eqref{eq17}, \eqref{eq18}, and the asymptotic properties of the function $F^{(1)}$. With $\varphi= \arg(\sigma z)$, $N\sim |\sigma z|$ and fixed $n$, F.~W.~J.~Olver \cite{Olver1991} showed that $F^{(1)}$ possesses the following asymptotic behaviour for large $|\sigma z|$:
\begin{equation}\label{eq52}
\frac{\e^{ - \sigma z} }{2\pi \im z^{N - 1 - n} }F^{(1)} \!\left( z;\!\begin{array}{c}
   N - n \! \\
   \sigma  \! \\
\end{array} \right) = - \frac{\im \e^{(\pi  - \varphi )\im N} }{1 + \e^{ -\varphi\im } }\frac{\e^{ - \sigma z - \left| \sigma z\right|} }{\sqrt {2\pi \left| \sigma z\right|} } ( 1 + \O(|\sigma z|^{-1}))
\end{equation}
when $|\varphi| \le \pi-\delta<\pi$, and
\begin{equation}\label{eq40}
\frac{\e^{ - \sigma z} }{2\pi \im z^{N - 1 - n} }F^{(1)} \!\left( z;\!\begin{array}{c}
   N - n \! \\
   \sigma \! \\
\end{array} \right) = \frac{1}{2}\erfc\Big(  c(\varphi ) \sqrt{\tfrac{1}{2}| \sigma z |}   \Big) + \mathcal{O}\bigg( \bigg| \frac{\e^{ - \frac{1}{2}| \sigma z |c^2 (\varphi )} }{\sqrt { \left| \sigma z\right|}} \bigg| \bigg),
\end{equation}
when $-\pi<-\pi+\delta \le \varphi \le 3\pi-\delta<3\pi$, with a conjugate behaviour in the sector $-3\pi<-3\pi+\delta \le \varphi \le \pi-\delta<\pi$. The quantity $c(\varphi )$ is defined by
\begin{equation}\label{eq48}
\frac{1}{2}c^2 (\varphi ) = 1 + \im(\varphi  - \pi ) - \e^{\im (\varphi  - \pi )},
\end{equation}
and corresponds to the branch of $c(\varphi )$ whose behaviour is
\begin{equation}\label{eq19}
c(\varphi) = -(\varphi-\pi) - \frac{\im}{6}(\varphi-\pi)^2 + \frac{1}{36}(\varphi-\pi)^3 + \ldots 
\end{equation}
near $\varphi=\pi$. Graphs of $\frac{1}{2}c^2 (\varphi )$ and $c(\varphi)$ are given in the paper by Olver \cite[Figs. 3 and 4]{Olver1991}.

We now observe that the main contributions to $R_N (z)$ and $\widetilde R_N (z)$ arise from the first sums in \eqref{eq17} and \eqref{eq18}, respectively. Neglecting error terms and retaining only the leading term in the expansion \eqref{eq19}, we find that as $z\to \infty$ within a region centred on the positive imaginary axis of angular width $\O(|z|^{-1/2})$,
\begin{equation}\label{eq20}
R_{N} (z) \sim \e^{2\pi \im z} \sum\limits_{m = 0}^{M-1} (-1)^m \frac{\gamma_m}{z^m}  \times \frac{1}{2}\erfc\left( \left(\tfrac{\pi}{2}-\theta\right)\sqrt{\pi |z|} \right)
\end{equation}
and
\begin{equation}\label{eq21}
\widetilde R_{N} (z) \sim  - \e^{2\pi \im z} \sum\limits_{m = 0}^{M-1} \frac{\gamma_m}{z^m}  \times \frac{1}{2}\erfc\left( \left(\tfrac{\pi}{2}-\theta\right)\sqrt{\pi |z|} \right)
\end{equation}
with $\theta =\arg z$, provided $M\geq 1$ is fixed. The effect of the complementary error function in \eqref{eq20} and \eqref{eq21} is to switch on the leading terms of the series pre-factored by the exponential $\e^{2\pi \im z}$ in \eqref{eq12} and \eqref{eq13} in the manner described by Berry as $\theta$ increases through $\frac{\pi}{2}$. Note that in Berry's original formulation the argument of the complementary error function is replaced by the approximate quantity
\[
\left(\tfrac{\pi}{2}-\theta\right)\sqrt{\pi |z|}  \approx \frac{\Im (2\pi \im z)}{\sqrt {2\Re (-2\pi \im z)}}.
\]
A similar result may be obtained for the Stokes line $\arg z =-\frac{\pi}{2}$, where the dominant contributions to $R_N (z)$ and $\widetilde R_N (z)$ come from the second sums in \eqref{eq17} and \eqref{eq18}.

It is clear from the estimates \eqref{eq22} that the sub-subdominant contributions (corresponding to $k \ge 2$ in \eqref{eq12} and \eqref{eq13}) cannot be detected by truncating the asymptotic expansions \eqref{eq9} and \eqref{eq10} near their least terms and keeping the number of terms in the re-expansions fixed. We expect that the contributions scaled by the factors $\e^{\pm 4\pi \im z}$ are born from the remainders of the optimally truncated Level 1 hyperasymptotic expansions \eqref{eq17} and \eqref{eq18}. That this is indeed the case may be verified numerically and graphically as follows. Optimal truncation is achieved when $N = \lfloor 4\pi |z|\rfloor$ and $M = \lfloor 2\pi |z|\rfloor$.\footnote{Generally, the optimal truncation indices at Level $k$ form a sequence in which the difference between consecutive terms is about $2\pi|z|$. This is because the difference between the exponents of two consecutive subdominant exponentials in \eqref{eq12} and \eqref{eq13} is precisely $\pm 2\pi \im z$. See \cite[\S4.1]{BHNOD2018} for truncation schemes in general.} The modified Stokes multipliers $\mathcal{S}_{N,M}^{(2)}$ and $\widetilde{\mathcal{S}}_{N,M}^{(2)}$ associated with the small exponential $\e^{4\pi \im z}$ for the Stokes line $\arg z = \frac{\pi}{2}$ are then defined exactly by the identities
\[
R_{N,M} (z) = \mathcal{S}_{N,M}^{(2)} \e^{4\pi \im z} \quad \text{ and } \quad \widetilde R_{N,M} (z) =  - \widetilde{\mathcal{S}}_{N,M}^{(2)} \e^{4\pi \im z} .
\]
Both $\mathcal{S}_{N,M}^{(2)}=\mathcal{S}_{N,M}^{(2)}(\theta)$ and $\widetilde{\mathcal{S}}_{N,M}^{(2)}=\widetilde{\mathcal{S}}_{N,M}^{(2)}(\theta)$ are piecewise continuous functions of $|z|$, and continuous functions of $\theta = \arg z$. The real parts of these modified Stokes multipliers with $|z| = 5$ are plotted against $\arg z$ on Figure \ref{Figure1}. The corresponding imaginary parts are negligible in magnitude. It is seen that $\mathcal{S}_{N,M}^{(2)}$ increases smoothly from approximately $0$ to approximately $1$ as $\theta$ passes continuously from values somewhat below $\frac{\pi}{2}$ to values somewhat above $\frac{\pi}{2}$. On the Stokes line $\theta  =\frac{\pi}{2}$, $\mathcal{S}_{N,M}^{(2)} \approx \frac{3}{8}$, in agreement with the discontinuous treatment \eqref{eq23}. Similarly, if $\theta$ changes from values somewhat below $\frac{\pi}{2}$ to values somewhat above $\frac{\pi}{2}$, then $\widetilde{\mathcal{S}}_{N,M}^{(2)}$ increases from approximately $0$ to approximately $\frac{1}{8}$ (when $\theta =\frac{\pi}{2}$) and then soon decreases to approximately $0$ again. This is consistent with the result \eqref{eq24} above. A similar analysis may be carried out for the other Stokes line $\arg z =-\frac{\pi}{2}$.

\begin{figure}[ht]
    \begin{subfigure}[c]{0.46\textwidth}
			\centering
			\includegraphics[width=\textwidth]{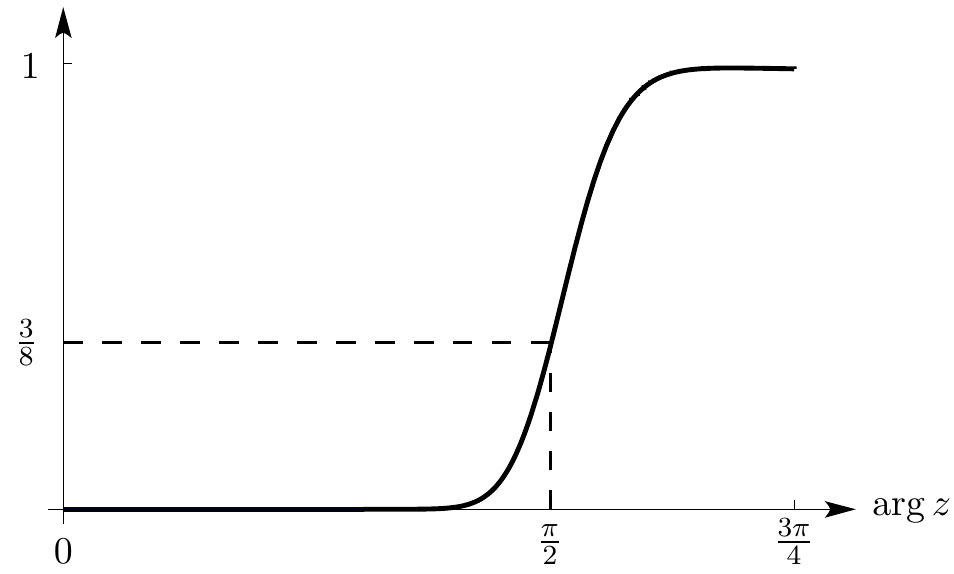}
      \caption{}
    \end{subfigure}
    \hfill
    \begin{subfigure}[c]{0.49\textwidth}
		\centering 
      \includegraphics[width=\textwidth]{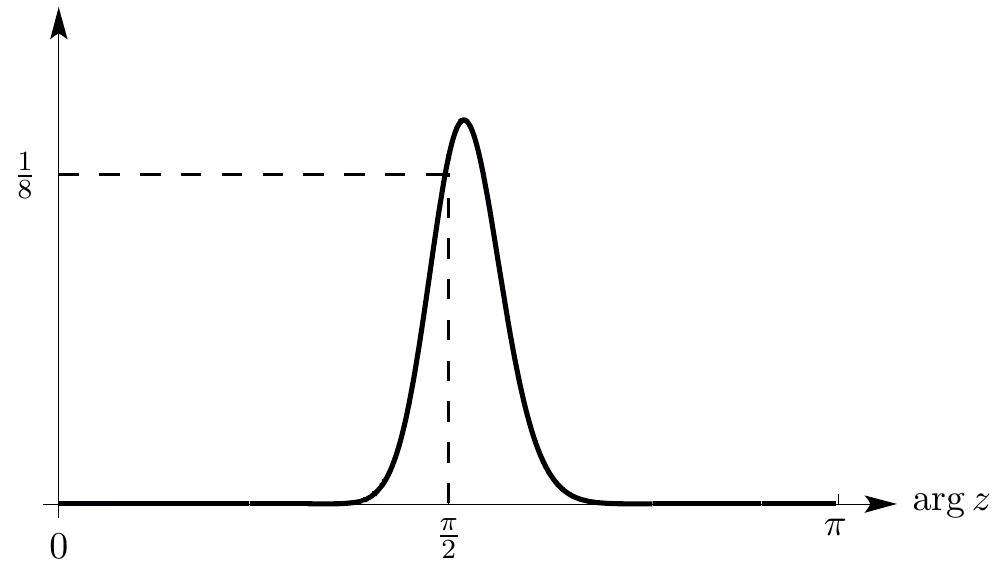}
      \caption{}
    \end{subfigure}
    \caption{(a) The behaviour of $\Re (\mathcal{S}_{N,M}^{(2)})$ with $|z| = 5$, $N = \lfloor 4\pi |z|\rfloor = 62$ and $M = \lfloor 2\pi |z|\rfloor = 31$ in the phase range $0 < \arg z < \frac{3\pi}{4}$. (b) The behaviour of $\Re (\widetilde{\mathcal{S}}_{N,M}^{(2)})$ with $|z| = 5$, $N = \lfloor 4\pi |z|\rfloor = 62$ and $M = \lfloor 2\pi |z|\rfloor = 31$ in the phase range $0 < \arg z < \pi$.}
		\label{Figure1}
  \end{figure}

To obtain a rigorous analytical description of these smooth transitions, it is necessary to go to the second stage of hyperasymptotics.

\subsection{Level 2 hyperasymptotics} The Level 2 hyperasymptotic expansions are derived by re-expanding the Level 1 expansions. Thus, in the integral representation \eqref{eq25} for the remainder $R_{N,M}(z)$ we substitute \eqref{eq17} (with $M$ in place of $N$, and $K$ in place of $M$) into $R_M (t)$. We obtain, using the definitions \eqref{eq34}, \eqref{eq27} and term-by-term integration, the re-expansion
\begin{gather}\label{eq53}
\begin{split}
R_{N,M} (z) = \; & \frac{1}{(2\pi \im)^2 }\frac{1}{z^{N - 1}}\sum\limits_{k = 0}^{K - 1} ( - 1)^k \gamma _k F^{(2)} \!\left( z;\!\begin{array}{c}
    N - M + 1,\!\!\!\!  \\ 
    2\pi \e^{\frac{\pi }{2}\im},\!\!\!\! \\
\end{array}\begin{array}{c}
   M - k  \\ 
   2\pi \e^{\frac{\pi }{2}\im}  \\
\end{array} \right)  
\\ & + \frac{1}{(2\pi \im)^2 }\frac{1}{z^{N - 1}}\sum\limits_{k = 0}^{K - 1} ( - 1)^k \gamma _k F^{(2)} \!\left( z;\!\begin{array}{c}
    N - M + 1,\!\!\!\!  \\ 
    2\pi \e^{-\frac{\pi }{2}\im} ,\!\!\!\!  \\
\end{array}\begin{array}{c}
   M - k  \\ 
   2\pi \e^{-\frac{\pi }{2}\im}  \\
\end{array} \right)
\\ & + \frac{1}{(2\pi \im)^2 }\frac{1}{z^{N - 1} }\sum\limits_{k = 0}^{K - 1} ( - 1)^k \gamma _k F^{(2)} \!\left( z;\!\begin{array}{c}
    N - M + 1,\!\!\!\!  \\ 
    2\pi \e^{\frac{\pi }{2}\im},\!\!\!\!  \\
\end{array}\begin{array}{c}
   M - k  \\ 
   2\pi \e^{-\frac{\pi}{2}\im} \\
\end{array} \right)
\\ & + \frac{1}{(2\pi \im)^2}\frac{1}{z^{N - 1}}\sum\limits_{k = 0}^{K - 1} ( - 1)^k \gamma _k F^{(2)} \!\left( z;\!\begin{array}{c}
    N - M + 1,\!\!\!\!  \\ 
    2\pi \e^{-\frac{\pi}{2}\im},\!\!\!\!  \\
\end{array}\begin{array}{c}
   M - k \\ 
   2\pi \e^{\frac{\pi}{2}\im}  \\
\end{array} \right) 
 + R_{N,M,K} (z),
\end{split}
\end{gather}
with
\[
R_{N,M,K} (z) = \frac{1}{2\pi \im}\frac{1}{z^N}\int_0^{[\pi/2^-]} \frac{\e^{2\pi \im t} t^{N - 1} }{1 - t/z}R_{M,K}(t)\d t  - \frac{1}{2\pi \im}\frac{1}{z^N}\int_0^{[-\pi/2^+]} \frac{\e^{ - 2\pi \im t} t^{N - 1}  }{1 - t/z}R_{M,K}(t)\d t,
\]
provided $K<M<N$. The corresponding expansion for $\widetilde R_{N,M} (z)$ can be derived in a similar manner, with the exception that special attention should be paid when identifying the various $F^{(2)}$ hyperterminants. In the instance that two singulants are identical, the argument of the function $F^{(1)}$ arising from the application of \eqref{eq17} will lie outside the principal domain of definition specified in \eqref{eq31}. In these cases, the application of the connection formula \eqref{eq49}, with $m=1$, becomes necessary. Then we find, after some algebraic computation, that the required re-expansion is
\begin{gather}\label{eq54}
\begin{split}
\widetilde R_{N,M} (z) = & -\frac{1}{2\pi \im}\frac{1}{z^{N - 1}}\sum\limits_{k = 0}^{K - 1} \gamma _k  F^{(1)} \!\left( z;\!\begin{array}{c}
    N - k\! \\ 
    4\pi \e^{\frac{\pi }{2}\im}\! \\
\end{array} \right) \\ & + \frac{1}{(2\pi \im)^2 }\frac{1}{z^{N - 1}}\sum\limits_{k = 0}^{K - 1}  \gamma _k F^{(2)} \!\left( z;\!\begin{array}{c}
     N - M + 1,\!\!\!\!   \\ 
     2\pi \e^{\frac{\pi }{2}\im} ,\!\!\!\! \\
\end{array}\begin{array}{c}
    M - k \\ 
    2\pi \e^{\frac{\pi }{2}\im}  \\
\end{array} \right) 
\\ &  + \frac{1}{2\pi \im}\frac{1}{z^{N - 1}}\sum\limits_{k = 0}^{K - 1} \gamma _k F^{(1)} \!\left( z;\!\begin{array}{c}
    N - k\! \\ 
    4\pi \e^{ - \frac{\pi }{2}\im}\! \\
   \end{array} \right) 
 \\ & + \frac{1}{(2\pi \im)^2}\frac{1}{z^{N - 1}}\sum\limits_{k = 0}^{K - 1} \gamma _k F^{(2)} \!\left( z;\!\begin{array}{c}
     N - M + 1,\!\!\!\!   \\ 
     2\pi \e^{-\frac{\pi}{2}\im} ,\!\!\!\! \\
\end{array}\begin{array}{c}
    M - k \\ 
    2\pi \e^{-\frac{\pi }{2}\im} \\
\end{array} \right)
\\ & + \frac{1}{(2\pi \im)^2}\frac{1}{z^{N - 1}}\sum\limits_{k = 0}^{K - 1} \gamma _k F^{(2)} \!\left( z;\!\begin{array}{c}
     N - M + 1,\!\!\!\! \\ 
    2\pi \e^{\frac{\pi }{2}\im} ,\!\!\!\!  \\
\end{array}\begin{array}{c}
    M - k \\ 
    2\pi \e^{-\frac{\pi }{2}\im} \\
\end{array} \right)
\\ & + \frac{1}{(2\pi \im)^2}\frac{1}{z^{N - 1}}\sum\limits_{k = 0}^{K - 1} \gamma _k F^{(2)} \!\left( z;\!\begin{array}{c}
     N - M + 1,\!\!\!\!   \\ 
     2\pi \e^{-\frac{\pi }{2}\im} ,\!\!\!\! \\
\end{array}\begin{array}{c}
    M - k \\ 
    2\pi \e^{\frac{\pi }{2}\im}  \\
\end{array} \right) 
+ \widetilde R_{N,M,K} (z),
\end{split}
\end{gather}
where
\begin{align*}
\widetilde R_{N,M,K} (z) = & - \frac{1}{2\pi \im}\frac{1}{z^N}\int_0^{[\pi /2^-]} \frac{\e^{2\pi \im t} t^{N - 1} }{1 - t/z}R_{M,K}(t\e^{-\pi \im} )\d t \\ & + \frac{1}{2\pi \im}\frac{1}{z^N}\int_0^{[ - \pi /2^ +  ]} \frac{\e^{ - 2\pi \im t} t^{N - 1} }{1 - t/z}R_{M,K}(t\e^{\pi \im} )\d t,
\end{align*}
with the constraint that $K<M<N$. The remainder terms $R_{N,M,K} (z)$ and $\widetilde R_{N,M,K} (z)$ are defined initially for $|\arg z|<\frac{\pi}{2}$, but can be extended via analytic continuation to $|\arg z|<\pi$.

The presence of the $F^{(2)}$ hyperterminants with identical singulants in \eqref{eq53} and \eqref{eq54} indicates that higher-order Stokes phenomena occur on the rays $\arg z =\pm \frac{\pi}{2}$ (cf. \cite[\S3]{Howls2004}). We shall now provide a smooth interpretation of these phenomena by describing the birth of the subdominant contributions with the exponentially small pre-factor $\e^{4\pi \im z}$ in the neighbourhood of the positive imaginary axis. The procedure is similar to that presented above in connection with the leading subdominant exponential $\e^{2\pi \im z}$. Suppose that the hyperasymptotic expansions \eqref{eq17} and \eqref{eq18} are truncated near their least terms, e.g., the truncation indices are $N = \lfloor 4\pi |z|\rfloor$ and $M = \lfloor 2\pi |z|\rfloor$, respectively. With this assumption, it can be shown that for any fixed $K$,
\[
R_{N,M,K}(z),\; \widetilde R_{N,M,K}(z) =\begin{cases} \O(\left|z\right|^{-K}\e^{-4\pi\left|z\right|}) & \text{if }\; |\arg z|\leq\frac{\pi}{2}, \\[0.5em] \O(\left|z\right|^{-K}\e^{-4\pi\Im(z)}) & \text{if }\; \frac{\pi}{2} < \arg z \leq \pi-\delta < \pi, \end{cases}
\]
as $z\to \infty$, uniformly with respect to $\arg z$ (cf. \cite[\S4.3]{BHNOD2018}). The behaviour of $R_{N,M}(z)$ and $\widetilde R_{N,M}(z)$ for large $z$ in the vicinity of the line $\arg z = \frac{\pi}{2}$ can now be inferred from \eqref{eq53}, \eqref{eq54}, and the asymptotic properties of the hyperterminants $F^{(1)}$ and $F^{(2)}$. With $\varphi= \arg(\sigma z)$, $N\sim |\sigma z|$ and fixed non-negative integer $n$, we show in Theorem \ref{maintheorem} that $F^{(2)}$ admits the following asymptotic behaviour for large $|\sigma z|$:
\begin{multline*}
\frac{\e^{-2\sigma z}}{(2\pi \im)^2 z^{2N-2-n}} F^{(2)} \!\left( z;\!\begin{array}{c}
     N,\!\!\!\!   \\ 
     \sigma ,\!\!\!\! \\
\end{array}\begin{array}{c}
    N-n\! \\ 
    \sigma \!\\
\end{array} \right) = \frac{1}{4}\erfc\Big(  c(\varphi ) \sqrt{| \sigma z |}   \Big)+\frac{1}{8}\erfc^2\Big(  c(\varphi ) \sqrt{\tfrac{1}{2}| \sigma z |}   \Big) \\ + \begin{dcases}
    \O\bigg( \bigg| \cfrac{\e^{ - \left| \sigma z \right|c^2 (\varphi )} }{\sqrt{|\sigma z|} } \bigg| \bigg), & \text{if }\; -\pi < -\pi + \delta \le \varphi \le \pi, \\
    \O\bigg( \bigg| \cfrac{\e^{ - \frac{1}{2}\left| \sigma z \right|c^2 (\varphi )} }{\sqrt{|\sigma z|} } \bigg| \bigg), & \text{if }\; \pi < \varphi \le 3\pi - \delta < 3\pi,
  \end{dcases}
\end{multline*}
with a conjugate behaviour in the sector $-3\pi<-3\pi+\delta \le \varphi \le \pi-\delta<\pi$. The quantity $c(\varphi )$ is defined by \eqref{eq48} and \eqref{eq19}. To estimate the $F^{(2)}$ hyperterminants with ``mixed" singulants, we use
\[
\frac{\e^{-2\sigma z}}{(2\pi \im)^2 z^{2N-2-n}} F^{(2)} \!\left( z;\!\begin{array}{c}
     N,\!\!\!\!   \\ 
     \sigma ,\!\!\!\! \\
\end{array}\begin{array}{c}
    N-n\! \\ 
    \sigma \e^{\pm\pi\im} \!\\
\end{array} \right)= \begin{dcases}
    \O\bigg( \bigg| \cfrac{\e^{ - \left| \sigma z \right|c^2 (\varphi )} }{\sqrt{|\sigma z|} } \bigg| \bigg), & \text{if }\; |\varphi| \le \pi, \\
    \O\bigg( \bigg| \cfrac{\e^{ - \frac{1}{2}\left| \sigma z \right|c^2 (\varphi )} }{\sqrt{|\sigma z|} } \bigg| \bigg), & \text{if }\; \pi < |\varphi| \le 2\pi - \delta < 2\pi \end{dcases}
\]
(see Proposition \ref{prop}). Accordingly, the main contribution to $R_{N,M} (z)$ comes from the first sum in \eqref{eq53}. Neglecting error terms and keeping only the leading term in the expansion \eqref{eq19}, we find that
\[
R_{N,M} (z) \sim \e^{4\pi \im z} \sum\limits_{k = 0}^{K-1} (-1)^k \frac{\gamma_k}{z^k}  \times \left(\frac{1}{4}\erfc\left( \left(\tfrac{\pi}{2}-\theta\right)\sqrt{2\pi |z|} \right)+\frac{1}{8}\erfc^2\left( \left(\tfrac{\pi}{2}-\theta\right)\sqrt{\pi |z|} \right)\right)
\]
for large $z$ such that $\theta=\arg z= \frac{\pi}{2}+\O(|z|^{-1/2})$, and with $K\geq 1$ held fixed. The combination of the $\erfc$ functions increases rapidly but smoothly from being exponentially small for $\theta < \frac{\pi}{2}$, to $\frac{3}{8}$ when $\theta = \frac{\pi}{2}$, and finally to almost $1$ when $\theta > \frac{\pi}{2}$. This is in complete accordance with the numerical results presented in the preceding subsection. The treatment in the neighbourhood of $\arg z = -\frac{\pi}{2}$ is similar, with the third sum in \eqref{eq53} then controlling the dominant behaviour of $R_{N,M} (z)$. Similarly, the transitional behaviour of $\widetilde R_{N,M} (z)$ from $\theta \le \frac{\pi}{2}$ to $\theta \ge \frac{\pi}{2}$ is essentially described by the first two sums in \eqref{eq54}. Therefore, 
\[
\widetilde R_{N,M} (z) \sim -\e^{4\pi \im z} \sum\limits_{k = 0}^{K-1} \frac{\gamma_k}{z^k}  \times \left(\frac{1}{4}\erfc\left( \left(\tfrac{\pi}{2}-\theta\right)\sqrt{2\pi |z|} \right)-\frac{1}{8}\erfc^2\left( \left(\tfrac{\pi}{2}-\theta\right)\sqrt{\pi |z|} \right)\right)
\]
for large $z$ such that $\theta= \frac{\pi}{2}+\O(|z|^{-1/2})$, and with fixed $K\geq 1$. The expression in the large parentheses increases rapidly but smoothly from approximately $0$ for $\theta < \frac{\pi}{2}$, to $\frac{1}{8}$ when $\theta = \frac{\pi}{2}$, and then drops down to $0$ again when $\theta > \frac{\pi}{2}$. This agrees with the numerical observation discussed in the previous subsection. It is thus seen that the role of the $F^{(1)}$ functions in the first line of \eqref{eq54} is to cancel the contributions switched on by the corresponding $F^{(2)}$ functions in the second line of \eqref{eq54} across the line $\arg z = \frac{\pi}{2}$. A similar result is obtained for the Stokes line $\arg z = -\frac{\pi}{2}$, where the third and fourth sums in \eqref{eq54} are dominant.

This process can be continued to produce a sequence of re-expanded remainder terms, each of which is exponentially smaller than its predecessor. The contributions scaled by the pre-factors $\e^{\pm 2\pi \im k z}$ are then born from the remainders of the optimally truncated Level $(k-1)$ hyperasymptotic expansions. Note that, in contrast to the case of $\log \Gamma(z)$, the smooth transition of successively smaller exponential contributions can only be demonstrated one at a time since each time the level is increased, the number of terms at each previous level in the hyperasymptotic expansions increases by approximately $2\pi|z|$.

\section{Howls' conjecture}\label{section4}

Howls (C.~J.~Howls, personal communication, December, 2017) conjectured the following formal power series identity involving hyperterminants:
\begin{equation}\label{eq26}
\exp \left( \sum\limits_{k = 1}^\infty \frac{(2\pi \im)^{k-1}}{k} F^{(1)} \!\left( z;\!\begin{array}{c}
   kN - k + 1 \! \\
   k\sigma \! \\
\end{array} \right)t^k \right) = \sum\limits_{m = 0}^\infty F^{(m)}\!\left( z;\! \begin{array}{c}
   N ,\,\ldots,\!\!\!\!  \\ 
   \,\sigma ,\,\,\ldots,\!\!\!\!  \\
\end{array}\begin{array}{c}
   N \!\!  \\
   \sigma  \!\!  \\
\end{array} \right)  t^m .
\end{equation}
This rather unexpected identity is suggested by the exponentially improved asymptotic expansion of $\log \Gamma(z)$ \cite[\S6.4.2]{Paris2001} and the hyperasymptotic expansion (to all levels) of $\Gamma(z)$. Our aim in this section is to rigorously interpret and verify this conjecture. To this end, it proves convenient to introduce the complete Bell polynomials ${\bf Y}_m (y_1 , \ldots ,y_m )$ of $m$ complex variables $y_1 , \ldots ,y_m$. They may be defined by
\begin{equation}\label{eq39}
{\bf Y}_m (y_1 , \ldots ,y_m ) = \sum\limits_{\pi(m)}\prod\limits_{j = 1}^m \frac{y_j^{k_j }}{k_j!},
\end{equation}
where $\pi(m)$ denotes a partition of $m$, usually denoted by $1^{k_1}2^{k_2}\cdots m^{k_m}$, with $k_1+2k_2+\ldots+m k_m=m$; $k_j\geq 0$ being the number of parts of size $j$. By convention, ${\bf Y}_0=1$. The complete Bell polynomials satisfy the generating function identity 
\begin{equation}\label{eq28}
\exp \left(\sum\limits_{k = 0}^\infty y_k t^k \right) = \sum\limits_{m = 0}^\infty {\bf Y}_m (y_1 , \ldots ,y_m )t^m .
\end{equation}
It should be noted that our notation differs from that used commonly in the literature. For example, the polynomials $Y_m$ discussed in \cite{Bell1934} or \cite[\S5.2]{Riordan1968} are related to ours via $Y_m (1!y_1 , \ldots ,m!y_m ) = m!{\bf Y}_m (y_1 , \ldots ,y_m )$. In view of \eqref{eq26} and \eqref{eq28}, we may reformulate Howls' conjecture as
\begin{equation}\label{eq29}
F^{(m)}\!\left( z;\! \begin{array}{c}
   N ,\,\ldots,\!\!\!\!  \\ 
   \,\sigma ,\,\,\ldots,\!\!\!\!  \\
\end{array}\begin{array}{c}
   N \!\!  \\
   \sigma  \!\!  \\
\end{array} \right) = {\bf Y}_m \left( F^{(1)} \!\left( z;\!\begin{array}{c}
   N \!  \\
   \sigma \!  \\
\end{array} \right), \ldots ,\frac{(2\pi \im)^{m - 1}}{m}F^{(1)}\! \left( z;\!\begin{array}{c}
   mN -m+ 1 \! \\
   m\sigma \! \\
\end{array} \right) \right).
\end{equation}
Notice that this identity, unlike \eqref{eq26}, does not suffer from possible divergence issues. We can now formulate the following theorem, which is the rigorous form of Howls' conjecture.

\begin{theorem} Let $N$ be a complex number such that $\Re(N)>1$, and let $\sigma$ be an arbitrary element of $\widehat{\mathbb C}$. Then the identity \eqref{eq29} holds for any $m\geq 0$ and for all values of $\arg z$.
\end{theorem}

\begin{proof} For the sake of brevity, let us denote the right-hand side of \eqref{eq29} by $f^{(m)} (z;N,\sigma )$. Thus, we are required to prove that
\begin{equation}\label{eq30}
f^{(m)} (z;N,\sigma )=
F^{(m)}\!\left( z;\! \begin{array}{c}
   N ,\,\ldots,\!\!\!\!  \\ 
   \,\sigma ,\,\,\ldots,\!\!\!\!  \\
\end{array}\begin{array}{c}
   N \!\!  \\
   \sigma  \!\!  \\
\end{array} \right)
\end{equation}
for all $m\geq 0$. We proceed by induction on $m$. It is readily seen that the equality \eqref{eq30} holds when $m=0$ or $1$. Assume that it holds for $0,1,\ldots,m-1$ ($m\geq 1$). Employing the recurrence relation
\[
{\bf Y}_m (y_1 , \ldots ,y_m ) = \frac{1}{m}\sum\limits_{k = 1}^m k y_k {\bf Y}_{m - k} (y_1 , \ldots ,y_{m - k} ) 
\]
(cf. \cite[Eq. (3), \S5.2]{Riordan1968}) and the induction hypothesis, we obtain
\[
f^{(m)} (z;N,\sigma ) = \frac{1}{m}\sum\limits_{k = 1}^m  (2\pi \im)^{k - 1} F^{(1)}\! \left( z;\!\begin{array}{c}
   kN -k+ 1 \! \\
   k\sigma \! \\
\end{array} \right) F^{(m-k)}\!\left( z;\! \begin{array}{c}
   N ,\,\ldots,\!\!\!\!  \\ 
   \,\sigma ,\,\,\ldots,\!\!\!\!  \\
\end{array}\begin{array}{c}
   N \!\!  \\
   \sigma  \!\!  \\
\end{array} \right) .
\]
If we replace $z$ by $z\e^{-2\pi \im}$ in this equality, apply the connection formula \eqref{eq49} for each hyperterminant on the right-hand side, and use \eqref{eq30} with $m-1$ in place of $m$, we obtain after some simplification that
\begin{equation}\label{eq32}
f^{(m)} (z\e^{ - 2\pi \im} ;N,\sigma ) = f^{(m)} (z;N,\sigma ) - 2\pi \im\e^{\sigma z} z^{N - 1} F^{(m-1)}\!\left( z;\! \begin{array}{c}
   N ,\,\ldots,\!\!\!\!  \\ 
   \,\sigma ,\,\,\ldots,\!\!\!\!  \\
\end{array}\begin{array}{c}
   N \!\!  \\
   \sigma  \!\!  \\
\end{array} \right).
\end{equation}
Now let us assume, for simplicity, that $|\arg z|<\pi$ and $\arg \sigma =0$. Let $\gamma$ be the positively oriented key-hole contour depicted in Figure \ref{Figure2}. By Cauchy's integral theorem, we can write
\begin{figure}[!t]
	\centering
		\includegraphics[width=0.4\textwidth]{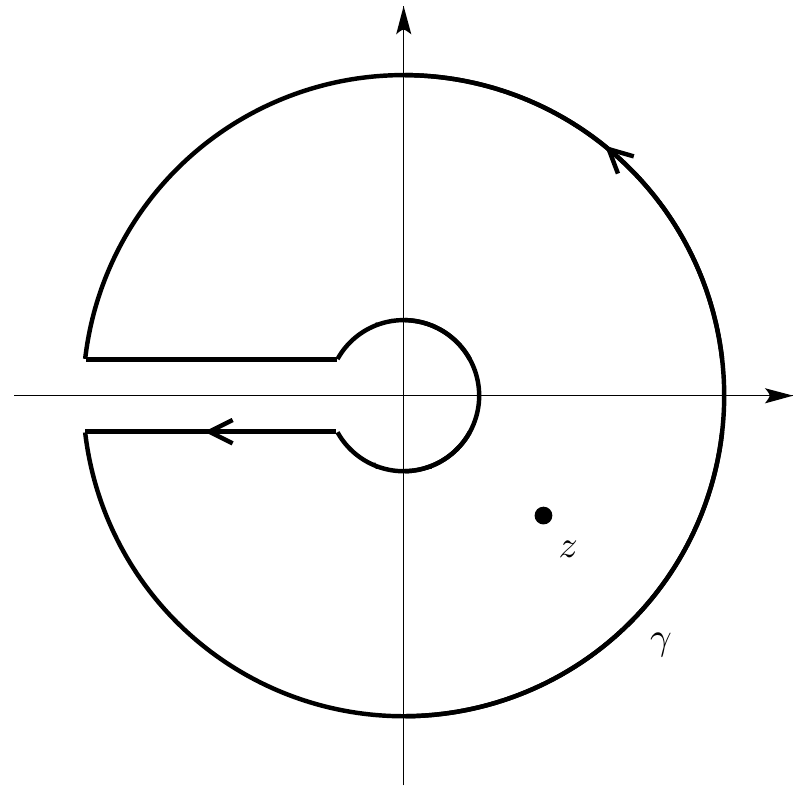}
		\caption{The contour of integration $\gamma$.}
		\label{Figure2}
\end{figure}
\[
f^{(m)} (z;N,\sigma ) = \frac{1}{2\pi \im}\oint_\gamma \frac{f^{(m)} (t;N,\sigma )}{t - z}\d t
\]
for all $z$ lying inside the contour $\gamma$. By virtue of the equality \eqref{eq33} and the standard asymptotic expansion of the incomplete gamma function (see, e.g., \cite[\href{https://dlmf.nist.gov/8.11.i}{\S8.11(i)}]{NIST:DLMF}), each hyperterminant on the right-hand side of \eqref{eq29} is $\O(|z|^{-1})$ as $z\to \infty$ in $|\arg z|\leq \pi$, provided all the other parameters held fixed. Consequently, for each $m\geq 2$ and fixed $N$, $\sigma$, $f^{(m)} (z;N,\sigma ) =\O(|z|^{-1})$ as $z\to \infty$ in the sector $|\arg z|\leq \pi$. The hyperterminants remain bounded as $z\to 0$ in $|\arg z|\leq \pi$, and thus so does $f^{(m)} (z;N,\sigma )$. Therefore, when the radius of the large circular portion of the contour $\gamma$ approaches $+\infty$, the integral along it tends to zero. Similarly, when the radius of the small circular arc tends to $0$, the integral along it tends to zero. If now we write $t=s \e^{\pm \pi \im}$ on the two rays, we find that 
\[
f^{(m)} (z;N,\sigma ) = \frac{1}{2\pi \im}\int_0^{ + \infty } \frac{f^{(m)} (s\e^{ - \pi \im} ;N,\sigma ) - f^{(m)} (s\e^{\pi \im} ;N,\sigma )}{s + z}\d s
\]
provided $|\arg z|<\pi$. Finally, the integrand may be expressed in terms of the hyperterminant $F^{(m-1)}$ by using \eqref{eq32}, and hence
\begin{align*}
f^{(m)} (z;N,\sigma ) &= \e^{\pi \im N} \int_0^{ + \infty } \frac{\e^{ - \sigma s} s^{N - 1} }{s+z} F^{(m-1)}\!\left( s\e^{\pi \im} ;\! \begin{array}{c}
   N ,\,\ldots,\!\!\!\!  \\ 
   \,\sigma ,\,\,\ldots,\!\!\!\!  \\
\end{array}\begin{array}{c}
   N \!\!  \\
   \sigma  \!\!  \\
\end{array} \right)\d s
\\ & =\int_0^{[\pi^-]} \frac{\e^{\sigma t} t^{N-1}}{z - t}F^{(m-1)}\!\left( t;\! \begin{array}{c}
   N ,\,\ldots,\!\!\!\!  \\ 
   \,\sigma ,\,\,\ldots,\!\!\!\!  \\
\end{array}\begin{array}{c}
   N \!\!  \\
   \sigma  \!\!  \\
\end{array} \right)\d t.
\end{align*}
It is seen from \eqref{eq34} and \eqref{eq27} that this integral is precisely the hyperterminant $F^{(m)}$ appearing on the right-hand side of the equality \eqref{eq30}. The restrictions on $\arg z$ and $\arg \sigma$ can now be dropped by appealing to analytic continuation.
\end{proof}

\begin{remark} Note that there is an alternative way to finish the proof of the theorem. Denote by $g^{(m)} (z;N,\sigma )$ the difference of the two sides of \eqref{eq30}. Assume for a moment that $N>2$. Then the functional equations \eqref{eq49} and \eqref{eq32}, Lemma \ref{lemma2} and Riemann's theorem on removable singularities together imply that $g^{(m)} (z;N,\sigma )$ is an entire function. From Lemma \ref{lemma1} we can assert that $g^{(m)} (z;N,\sigma ) =\O(|z|^{-1})$ as $z\to \infty$. Hence by Liouville's theorem, $g^{(m)} (z;N,\sigma )$ is identically zero. The restriction $N>2$ can be removed by analytic continuation.
\end{remark}

\section{Uniform asymptotic approximations for the hyperterminants}\label{section5}

For the gamma function and its reciprocal, the smoothing of the (higher-order) Stokes phenomenon is described in terms of $F^{(m)}$ hyperterminants of the form
\[
F^{(m)}\!\left( z;\! \begin{array}{c}
   N ,\,\ldots,\!\!\!\!  \\ 
   \,\sigma ,\,\,\ldots,\!\!\!\!  \\
\end{array}\begin{array}{c}
   N, \!\!\!  \\
   \sigma,  \!\!\!  \\
\end{array}\begin{array}{c}
   N-n \!\!  \\
   \sigma  \!\!  \\
\end{array} \right).
\]
Of most interest for our purpose is their behaviour when $N\sim |\sigma z| \gg 1$, $n$ is a fixed non-negative integer, and $\varphi=\arg(\sigma z)$ increases (resp. decreases) through $\pi$ (resp. $-\pi$). In these circumstances, it is seen from \eqref{eq40} and \eqref{eq19} that, when suitably normalised, $F^{(1)}$ possesses the property of changing rapidly, but smoothly, from being exponentially small to being approximately $1$ as $\varphi$ increases continuously through $\pi$. A similar behaviour is present near the ray $\varphi=-\pi$. In Theorem \ref{maintheorem} below we show that each of the higher-level terminants $F^{(m)}$, $m\ge 2$, exhibits this behaviour, which is, however, no longer described by a single complementary error function but rather by a multivariate polynomial in such functions with properly scaled arguments. That the leading-order asymptotics should take this form is, of course, no surprise, in view of the identity \eqref{eq29} which plays an essential role in our proof.

\begin{theorem}\label{maintheorem} Let $m$ be a positive integer and $\sigma$ be any element of $\widehat{\mathbb C}$. Let $N = |\sigma z|+\rho$, $\varphi=\arg(\sigma z)$, and define $c(\varphi)$ by \eqref{eq48} and \eqref{eq19}. Then for any fixed non-negative integer $n$ and large $|\sigma z|$,
\begin{equation}\label{eq41}
\begin{multlined}
\frac{\e^{ - m\sigma z} }{(2\pi \im)^m z^{mN - m-n}} F^{(m)}\!\left( z;\! \begin{array}{c}
   N ,\,\ldots,\!\!\!\!  \\ 
   \,\sigma ,\,\,\ldots,\!\!\!\!  \\
\end{array}\begin{array}{c}
   N, \!\!\!  \\
   \sigma,  \!\!\!  \\
\end{array}\begin{array}{c}
   N-n \!\!  \\
   \sigma  \!\!  \\
\end{array} \right) = \sum\limits_{\pi (m)} \prod\limits_{j = 1}^m \frac{1}{(2j)^{k_j } k_j !}\erfc^{k_j } \Big( c(\varphi )\sqrt {\tfrac{j}{2}\left| \sigma z \right|} \Big)  \\  + \begin{dcases}
    \O\bigg( \bigg| \cfrac{\e^{ - \frac{m}{2}\left| \sigma z \right|c^2 (\varphi )} }{\sqrt{|\sigma z|} } \bigg| \bigg), & \text{if }\; -\pi < -\pi + \delta \le \varphi \le \pi, \\
    \O\bigg( \bigg| \cfrac{\e^{ - \frac{1}{2}\left| \sigma z \right|c^2 (\varphi )} }{\sqrt{|\sigma z|} } \bigg| \bigg), & \text{if }\; \pi < \varphi \le 3\pi - \delta < 3\pi,
  \end{dcases}
\end{multlined}
\end{equation}
and
\begin{equation}\label{eq45}
\begin{multlined}
\frac{\e^{ - 2\pi \im mN} \e^{ - m\sigma z} }{(2\pi \im)^m z^{mN - m-n}} F^{(m)}\!\left( z;\! \begin{array}{c}
   N ,\,\ldots,\!\!\!\!  \\ 
   \,\sigma ,\,\,\ldots,\!\!\!\!  \\
\end{array}\begin{array}{c}
   N, \!\!\!  \\
   \sigma,  \!\!\!  \\
\end{array}\begin{array}{c}
   N-n \!\!  \\
   \sigma  \!\!  \\
\end{array} \right) = \sum\limits_{\pi (m)} \prod\limits_{j = 1}^m \frac{(-1)^{k_j}}{(2j)^{k_j } k_j !}\erfc^{k_j } \Big( \overline{c(-\varphi)}\sqrt {\tfrac{j}{2}\left| \sigma z \right|} \Big) \\ + \begin{dcases}
    \O\bigg( \bigg| \cfrac{\e^{ - \frac{m}{2}\left| \sigma z \right|c^2 (-\varphi )} }{\sqrt{|\sigma z|} } \bigg| \bigg), & \text{if }\; -\pi  \le \varphi \le \pi- \delta <\pi, \\
    \O\bigg( \bigg| \cfrac{\e^{ - \frac{1}{2}\left| \sigma z \right|c^2 (-\varphi )} }{\sqrt{|\sigma z|} } \bigg| \bigg), & \text{if }\; -3\pi < -3\pi + \delta \leq \varphi < -\pi,
  \end{dcases}
\end{multlined}
\end{equation}
uniformly with respect to $\varphi$ and bounded real values of $\rho$. The index $\pi(m)$ runs through all partitions of $m$ into non-negative parts, i.e., over all non-negative integer solutions of the equation $k_1+2k_2+\ldots+m k_m=m$.
\end{theorem}

\begin{remark} If $m=1$, then the $\varphi$-intervals of validity for \eqref{eq41} and \eqref{eq45} are maximal (see \cite[Theorem 1]{Olver1991}). We expect this property to hold generally for all $m\geq 1$, but we shall not pursue a proof of this claim here.
\end{remark}

Successive hyperasymptotic re-expansions of $\Gamma^\ast(z)$ and its reciprocal also involve hyperterminants $F^{(m)}$ with ``mixed" singulants. We expect that the contribution from such hyperterminants is of the same order of magnitude as the error terms in \eqref{eq41} and \eqref{eq45}. We therefore anticipate that the smooth transition of the higher-order Stokes discontinuities is essentially described by \eqref{eq41} and \eqref{eq45}. This is supported by the fact that \eqref{eq41} yields, to leading order, $\frac{1}{k!}\left(\frac{1}{2}\right)_k$ on the Stokes line $\varphi=\pi$, in agreement with \eqref{eq23}. In the following proposition, we verify this assertion in the particular case of $m=2$. We believe that the general case can be established using an appropriate inductive argument, but we leave it as an open question for further research.

\begin{proposition}\label{prop} Let $\sigma$ be any element of $\widehat{\mathbb C}$. Let $N = |\sigma z|+\rho$, $\varphi=\arg(\sigma z)$, and define $c(\varphi)$ by \eqref{eq48} and \eqref{eq19}. Then for any fixed non-negative integer $n$ and large $|\sigma z|$,
\begin{equation}\label{eq55}
\frac{\e^{-2\sigma z}}{(2\pi \im)^2 z^{2N-2-n}} F^{(2)} \!\left( z;\!\begin{array}{c}
     N,\!\!\!\!   \\ 
     \sigma ,\!\!\!\! \\
\end{array}\begin{array}{c}
    N-n\! \\ 
    \sigma \e^{\pm\pi\im} \!\\
\end{array} \right)= \begin{dcases}
    \O\bigg( \bigg| \cfrac{\e^{ - \left| \sigma z \right|c^2 (\varphi )} }{\sqrt{|\sigma z|} } \bigg| \bigg), & \text{if }\; |\varphi| \le \pi, \\
    \O\bigg( \bigg| \cfrac{\e^{ - \frac{1}{2}\left| \sigma z \right|c^2 (\varphi )} }{\sqrt{|\sigma z|} } \bigg| \bigg), & \text{if }\; \pi < |\varphi| \le 2\pi - \delta < 2\pi,
  \end{dcases}
\end{equation}
uniformly with respect to $\varphi$ and bounded real values of $\rho$.
\end{proposition}

To prove Theorem \ref{maintheorem} and Proposition \ref{prop} we first need to establish some lemmata.

\begin{lemma}\label{lemma1} Let $m$ be a positive integer and $N_1,\ldots,N_m$ be an arbitrary set of real numbers such that $N_k>1$ for $k=1,\ldots,m$. Let $\sigma$ be any element of $\widehat{\mathbb C}$. Then there exists a positive constant $c_m$, depending only on $m$, such that 
\begin{gather}\label{eq38}
\begin{split}
\left| F^{(m)} \!\left( z;\!\begin{array}{c}
   N_1 ,\!\!\!  \\
   \sigma ,\!\!\!  \\
\end{array}\begin{array}{c}
   \ldots ,\!\!\!  \\
    \ldots ,\!\!\!  \\
\end{array}\begin{array}{c}
   N_m \!\!  \\
   \sigma \!\! \\
\end{array} \right)\right| \le \; & c_{m} \frac{1}{|z|}\frac{\sqrt{N_m } \Gamma (N_m )}{|\sigma|^{N_m }}\prod\limits_{k = 1}^{m - 1} \frac{\sqrt{\Sigma_k} \Gamma(N_k  - 1)}{|\sigma|^{N_k- 1}} \\ & \times \begin{cases} 1 & \text{if }\; |\varphi|\leq \pi, \\[0.25em] \left| \cos \varphi \right|^{ - \Sigma_1} & \text{if }\; 
\pi < \varphi < \frac{3\pi}{2}, \end{cases}
\end{split}
\end{gather}
where $\varphi=\arg(\sigma z)$ and $\Sigma_k  = N_k +N_{k+1} +  \ldots  + N_m$.
\end{lemma}

\begin{proof} Throughout the proof, we shall use the following inequality:
\begin{equation}\label{eq35}
\left| 1 + \frac{t}{w} \right| \ge \begin{cases} 1 & \text{if } \;  |\arg w| \leq \frac{\pi}{2}, \\[0.25em] |\sin(\arg w)| & \text{if } \; \frac{\pi}{2} < |\arg w| < \pi,\end{cases}
\end{equation}
where $t > 0$. The proof of \eqref{eq35} is elementary and is therefore left to the reader.

We proceed by induction on $m$. The base case $m = 1$ was proved in \cite[Proposition B.1]{BHNOD2018}. Assume that the statement holds for $1,2,\ldots,m-1$ ($m \geq 2$). If $|\varphi|<\pi$, we can write, using \eqref{eq34} and \eqref{eq27}, that
\begin{equation}\label{eq36}
F^{(m)} \!\left( z;\!\begin{array}{c}
   N_1 ,\!\!\!  \\
   \sigma ,\!\!\!  \\
\end{array}\begin{array}{c}
   \ldots ,\!\!\!  \\
    \ldots ,\!\!\!  \\
\end{array}\begin{array}{c}
   N_m \!\!  \\
   \sigma \!\! \\
\end{array} \right) = \e^{\pi \im N_1 } \frac{1}{z}\frac{1}{\sigma^{N_1}}\int_0^{+\infty}\frac{\e^{-t} t^{N_1-1}}{1+t/(\sigma z)} F^{(m-1)}\!\left( \frac{t}{\sigma}\e^{\pi \im};\! \begin{array}{c}
   N_2 ,\!\!\!  \\
   \sigma ,\!\!\!  \\
\end{array}\begin{array}{c}
   \ldots ,\!\!\!  \\
    \ldots ,\!\!\!  \\
\end{array}\begin{array}{c}
   N_m \!\!  \\
   \sigma \!\! \\
\end{array} \right)\d t .
\end{equation}
Then by the induction hypothesis and the inequality \eqref{eq35}, we readily find that
\begin{align*}
 \left| F^{(m)} \!\left( z;\!\begin{array}{c}
   N_1 ,\!\!\!  \\
   \sigma ,\!\!\!  \\
\end{array}\begin{array}{c}
   \ldots ,\!\!\!  \\
    \ldots ,\!\!\!  \\
\end{array}\begin{array}{c}
   N_m \!\!  \\
   \sigma \!\! \\
\end{array} \right)\right| & \le c_{m - 1} \frac{1}{\left| z \right|}\frac{\sqrt {N_m } \Gamma (N_m )}{\left| \sigma  \right|^{N_m } }\frac{\Gamma (N_1  - 1)}{\left| \sigma  \right|^{N_1  - 1} }\prod\limits_{k = 2}^{m - 1} \frac{\sqrt {\Sigma_k} \Gamma (N_k  - 1)}{\left| \sigma  \right|^{N_k  - 1} } 
\\ & < c_{m - 1} \frac{1}{\left| z \right|}\frac{\sqrt {N_m } \Gamma (N_m )}{\left| \sigma  \right|^{N_m } }\prod\limits_{k = 1}^{m - 1} \frac{\sqrt {\Sigma_k } \Gamma (N_k  - 1)}{\left| \sigma  \right|^{N_k  - 1} } ,
\end{align*}
provided $|\varphi|\leq \frac{\pi}{2}$.

Consider now the phase range $\frac{\pi}{2} < \varphi \leq \pi$. We deform the contour of integration in \eqref{eq36} by rotating it through an acute angle $\alpha$. Thus, by appealing to Cauchy's theorem and analytic continuation, we have, for an arbitrary $0 < \alpha < \frac{\pi}{2}$, that
\begin{align*}
F^{(m)} \!\left( z;\!\begin{array}{c}
   N_1 ,\!\!\!  \\
   \sigma ,\!\!\!  \\
\end{array}\begin{array}{c}
   \ldots ,\!\!\!  \\
    \ldots ,\!\!\!  \\
\end{array}\begin{array}{c}
   N_m \!\!  \\
   \sigma \!\! \\
\end{array} \right) = \; & \e^{ \pi \im N_1} \frac{1}{z}\frac{1}{\sigma^{N_1}} \left( \frac{\e^{\im\alpha } }{\cos \alpha } \right)^{N_1 } \\ & \times  \int_0^{+\infty}\frac{\e^{-\frac{s\e^{\im\alpha }}{\cos\alpha}} s^{N_1-1}}{1+s\e^{\im\alpha }/(\sigma z \cos\alpha)} F^{(m-1)}\!\left( \frac{s\e^{\im(\alpha + \pi)}}{\sigma \cos\alpha};\! \begin{array}{c}
   N_2 ,\!\!\!  \\
   \sigma ,\!\!\!  \\
\end{array}\begin{array}{c}
   \ldots ,\!\!\!  \\
    \ldots ,\!\!\!  \\
\end{array}\begin{array}{c}
   N_m \!\!  \\
   \sigma \!\! \\
\end{array} \right)\d s
\end{align*}
when $\frac{\pi}{2} < \varphi \leq \pi$. Employing the inequality \eqref{eq35}, we then obtain the bound
\begin{align*}
\left| F^{(m)} \!\left( z;\!\begin{array}{c}
   N_1 ,\!\!\!  \\
   \sigma ,\!\!\!  \\
\end{array}\begin{array}{c}
   \ldots ,\!\!\!  \\
    \ldots ,\!\!\!  \\
\end{array}\begin{array}{c}
   N_m \!\!  \\
   \sigma \!\! \\
\end{array} \right)\right| \le \; &  \frac{1}{|z|}\frac{1}{|\sigma|^{N_1 }}\frac{1}{(\cos\alpha)^{N_1 }} \\ & \times \int_0^{+\infty} \e^{ - s} s^{N_1 - 1} \left| F^{(m-1)}\!\left( \frac{s\e^{\im(\alpha + \pi)}}{\sigma \cos\alpha};\!\begin{array}{c}
   N_2 ,\!\!\!  \\
   \sigma ,\!\!\!  \\
\end{array}\begin{array}{c}
   \ldots ,\!\!\!  \\
    \ldots ,\!\!\!  \\
\end{array}\begin{array}{c}
   N_m \!\!  \\
   \sigma \!\! \\
\end{array} \right) \right|\d s 
\\ & \times \begin{cases} 1 & \text{if }\; \frac{\pi}{2} < \varphi \leq \frac{\pi}{2} + \alpha, \\[0.25em] \csc (\varphi-\alpha) & \text{if }\; 
\frac{\pi}{2}+ \alpha < \varphi \leq \pi. \end{cases}
\end{align*} 
We can simplify this result further by using the induction hypothesis to deduce
\begin{gather}\label{eq37}
\begin{split}
\left| F^{(m)} \!\left( z;\!\begin{array}{c}
   N_1 ,\!\!\!  \\
   \sigma ,\!\!\!  \\
\end{array}\begin{array}{c}
   \ldots ,\!\!\!  \\
    \ldots ,\!\!\!  \\
\end{array}\begin{array}{c}
   N_m \!\!  \\
   \sigma \!\! \\
\end{array} \right)\right| \le \; &  c_{m - 1} \frac{1}{\left| z \right|}\frac{\sqrt {N_m } \Gamma (N_m )}{\left| \sigma  \right|^{N_m } }\frac{\Gamma (N_1  - 1)}{\left| \sigma  \right|^{N_1  - 1} }\prod\limits_{k = 2}^{m - 1} \frac{\sqrt {\Sigma_k} \Gamma (N_k  - 1)}{\left| \sigma  \right|^{N_k  - 1} } 
\\ & \times \frac{1}{(\cos \alpha )^{\Sigma_1- 1} }
\times \begin{cases} 1 & \text{if }\; \frac{\pi}{2} < \varphi \leq \frac{\pi}{2} + \alpha, \\[0.25em] \csc (\varphi-\alpha) & \text{if }\; 
\frac{\pi}{2}+ \alpha < \varphi \leq \pi. \end{cases}
\end{split}
\end{gather}
We now choose the value of $\alpha$ which minimises the right-hand side of this inequality when $\varphi=\pi$, namely $\alpha  = \arccot (\sqrt {\Sigma_1 - 1} )$. With this choice of $\alpha$, the factor in the second line of \eqref{eq37} may be bounded by $\sqrt {\e\, \Sigma_1}$ for all $\frac{\pi}{2} < \varphi \leq \pi$ (see the proof of Proposition B.1 in \cite{BHNOD2018}). A similar proof holds for the conjugate sector $-\pi \leq \varphi < -\frac{\pi}{2}$.

Finally, it remains to consider the sector $\pi <\varphi<\frac{3\pi}{2}$. The proof is based on the functional relation (cf. \eqref{eq49})
\begin{align*}
F^{(m)} \!\left( z;\!\begin{array}{c}
   N_1 ,\!\!\!  \\
   \sigma ,\!\!\!  \\
\end{array}\begin{array}{c}
   \ldots ,\!\!\!  \\
    \ldots ,\!\!\!  \\
\end{array}\begin{array}{c}
   N_m \!\!  \\
   \sigma \!\! \\
\end{array} \right) = \; & F^{(m)} \!\left( z\e^{-2\pi \im};\!\begin{array}{c}
   N_1 ,\!\!\!  \\
   \sigma ,\!\!\!  \\
\end{array}\begin{array}{c}
   \ldots ,\!\!\!  \\
    \ldots ,\!\!\!  \\
\end{array}\begin{array}{c}
   N_m \!\!  \\
   \sigma \!\! \\
\end{array} \right) \\& + 2\pi \im \e^{\sigma z} z^{N_1-1} F^{(m-1)} \!\left( z;\!\begin{array}{c}
   N_2 ,\!\!\!  \\
   \sigma ,\!\!\!  \\
\end{array}\begin{array}{c}
   \ldots ,\!\!\!  \\
    \ldots ,\!\!\!  \\
\end{array}\begin{array}{c}
   N_m \!\!  \\
   \sigma \!\! \\
\end{array} \right).
\end{align*}
The first term on the right-hand side can be estimated by applying the result of the previous paragraph. To estimate the second term, we use the induction hypothesis. Accordingly,
\begin{multline*}
 \left| F^{(m)} \!\left( z;\!\begin{array}{c}
   N_1 ,\!\!\!  \\
   \sigma ,\!\!\!  \\
\end{array}\begin{array}{c}
   \ldots ,\!\!\!  \\
    \ldots ,\!\!\!  \\
\end{array}\begin{array}{c}
   N_m \!\!  \\
   \sigma \!\! \\
\end{array} \right)\right| \le  \sqrt{\e}c_{m-1} \frac{1}{|z|}\frac{\sqrt{N_m } \Gamma (N_m )}{|\sigma|^{N_m }}\prod\limits_{k = 1}^{m - 1} \frac{\sqrt{\Sigma_k} \Gamma(N_k  - 1)}{|\sigma|^{N_k- 1}} \\ +\frac{2\pi \left| \sigma z \right|^{N_1  - 1} \e^{-\left| \sigma z \right|\left| \cos \varphi\right|}}{\sqrt {\Sigma _1 } \Gamma (N_1  - 1) }\frac{1}{\left| \cos \varphi  \right|^{\Sigma _2 } }c_{m - 1} \frac{1}{\left| z \right|}\frac{\sqrt {N_m } \Gamma (N_m )}{\left| \sigma  \right|^{N_m } }\prod\limits_{k = 1}^{m - 1} \frac{\sqrt {\Sigma _k } \Gamma (N_k  - 1)}{\left| \sigma  \right|^{N_k  - 1} } .
\end{multline*}
Notice that the quantity $r^M\e^{-a r}$, as a function of $r > 0$, takes its maximum value at $r = M/a$ when $a>0$ and $M > 0$. We therefore find that
\begin{align*}
\frac{2\pi \left| \sigma z \right|^{N_1  - 1} \e^{-\left| \sigma z \right|\left| \cos \varphi \right|}}{\sqrt {\Sigma _1 }\Gamma (N_1  - 1) }\frac{1}{\left| \cos \varphi  \right|^{\Sigma _2 } }  \le  \frac{\sqrt {2\pi (N_1  - 1)}}{\sqrt {\Sigma _1 } \Gamma^\ast (N_1  - 1)}\frac{1}{\left| \cos \varphi \right|^{\Sigma _1  - 1} } < \sqrt {2\pi } \frac{1}{\left| \cos \varphi \right|^{\Sigma _1 }}.
\end{align*}
The second inequality can be obtained from the fact that $\Gamma^\ast (M) \ge 1$ for any $M > 0$ (see, for instance, \cite[\href{https://dlmf.nist.gov/5.6.E1}{Eq. 5.6.1}]{NIST:DLMF}). This completes the proof of the lemma.
\end{proof}

We now define the value of the hyperterminant $F^{(m)}$ at the origin to be
\[
F^{(m)}\!\left( 0;\! \begin{array}{c}
   N_1 ,\,\ldots,\!\!\!\!  \\ 
   \,\sigma _1 ,\,\,\ldots,\!\!\!\!  \\
\end{array}\begin{array}{c}
   N_m \!\!  \\
   \sigma _m \!\!  \\
\end{array} \right)  = \mathop {\lim }\limits_{\substack{\left| z \right| \to 0\\ \left| \arg (\sigma_1 z) \right| < \pi }} F^{(m)}\!\left( z;\! \begin{array}{c}
   N_1 ,\,\ldots,\!\!\!\!  \\ 
   \,\sigma _1 ,\,\,\ldots,\!\!\!\!  \\
\end{array}\begin{array}{c}
   N_m \!\!  \\
   \sigma _m \!\!  \\
\end{array} \right),
\]
provided that this limit exists.

\begin{lemma}\label{lemma2} Let $m\geq 2$ be a positive integer and $N_1,\ldots,N_m$ be an arbitrary set of real numbers such that $N_1>2$ and $N_k>1$ for $k=2,\ldots,m$. Let $\sigma$ be any element of $\widehat{\mathbb C}$. Then
\[
\left| F^{(m)} \!\left( 0;\!\begin{array}{c}
   N_1 ,\!\!\!  \\
   \sigma ,\!\!\!  \\
\end{array}\begin{array}{c}
   \ldots ,\!\!\!  \\
    \ldots ,\!\!\!  \\
\end{array}\begin{array}{c}
   N_m \!\!  \\
   \sigma \!\! \\
\end{array} \right) \right| \le c_m \frac{\sqrt {N_m } \Gamma (N_m )}{\left| \sigma  \right|^{N_m }}\frac{\Gamma (N_1  - 2)}{\left| \sigma  \right|^{N_1  - 2} }\prod\limits_{k = 2}^{m - 1} \frac{\sqrt {\Sigma_k } \Gamma (N_k  - 1)}{\left| \sigma  \right|^{N_k  - 1} }
\]
where $c_m$ is the constant appearing in Lemma \ref{lemma1} and, as before, $\Sigma_k  = N_k +N_{k+1} +  \ldots  + N_m$.
\end{lemma}

\begin{proof} From \eqref{eq36}, we can infer that
\[
F^{(m)} \!\left( 0;\!\begin{array}{c}
   N_1 ,\!\!\!  \\
   \sigma ,\!\!\!  \\
\end{array}\begin{array}{c}
   \ldots ,\!\!\!  \\
    \ldots ,\!\!\!  \\
\end{array}\begin{array}{c}
   N_m \!\!  \\
   \sigma \!\! \\
\end{array} \right) = \e^{\pi \im N_1 } \frac{1}{\sigma^{N_1-1}}\int_0^{+\infty} \e^{-t} t^{N_1-2}F^{(m-1)}\!\left( \frac{t}{\sigma}\e^{\pi \im};\! \begin{array}{c}
   N_2 ,\!\!\!  \\
   \sigma ,\!\!\!  \\
\end{array}\begin{array}{c}
   \ldots ,\!\!\!  \\
    \ldots ,\!\!\!  \\
\end{array}\begin{array}{c}
   N_m \!\!  \\
   \sigma \!\! \\
\end{array} \right)\d t .
\]
The desired result now follows by estimating the right-hand side using the inequality \eqref{eq38}.
\end{proof}

\begin{proof}[Proof of Theorem \ref{maintheorem}] Throughout the proof, we will frequently make use of the following inequality:
\begin{equation}\label{eq42}
\left| \erfc(w) \right| \leq \begin{cases}
    K|\e^{-w^2}|, & \text{if }\; \Re (w)\geq 0, \\
    K, & \text{if }\; \Re(w)< 0,
  \end{cases}
\end{equation}
with an absolute constant $K>0$. This inequality can be verified by reference to the large-$w$ behaviour of $\erfc(w)$ (see, e.g., \cite[\href{https://dlmf.nist.gov/7.12.i}{\S7.12(i)}]{NIST:DLMF}) and the fact that $\erfc$ is an entire function. We leave the details to the interested reader.

We begin with the proof of \eqref{eq41}. The proof is by induction on $n$. For the base case $n=0$, we combine \eqref{eq29}, \eqref{eq39} and \eqref{eq40} to obtain
\begin{multline*}
\frac{\e^{ - m\sigma z} }{(2\pi \im)^m z^{mN - m}} F^{(m)}\!\left( z;\! \begin{array}{c}
   N ,\,\ldots,\!\!\!\!  \\ 
   \,\sigma ,\,\,\ldots,\!\!\!\!  \\
\end{array}\begin{array}{c}
   N \!\!  \\
   \sigma  \!\!  \\
\end{array} \right) \\ = \sum\limits_{\pi (m)} \prod\limits_{j = 1}^m \frac{1}{j^{k_j } k_j !}\left( \frac{1}{2}\erfc\Big( c(\varphi )\sqrt {\tfrac{j}{2}\left| \sigma z \right|} \Big) + \O\bigg(\bigg| \frac{\e^{ - \frac{j}{2}\left| \sigma z \right|c^2 (\varphi )} }{\sqrt {j\left| \sigma z \right|} }  \bigg| \bigg)\right)^{k_j }   .
\end{multline*}
In arriving at this expression, we made use of the fact that $k_1+2k_2+\ldots +m k_m=m$. Expanding the right-hand side using the binomial theorem gives \eqref{eq41} but with the error term
\begin{equation}\label{eq43}
\O(1)\sum\limits_{\pi (m)} \prod\limits_{j = 1}^m \frac{1}{j^{k_j } k_j !}\frac{\e^{ - \frac{jk_j }{2}\left| \sigma z \right|c^2 (\varphi )} }{\sqrt {j\left| \sigma z \right|} }\sum\limits_{r = 0}^{k_j - 1}\binom{k_j}{r}\frac{1}{2^r}\erfc^r\Big( c(\varphi )\sqrt {\tfrac{j}{2}\left| \sigma z \right|} \Big) \frac{\e^{\frac{jr}{2}\left| \sigma z \right|c^2 (\varphi )} }{(j\left| \sigma z \right|)^{(k_j  - 1 - r)/2}} .
\end{equation}
Consider first the case that $-\pi < -\pi + \delta \le \varphi \le \pi$. In this case $\Re(c(\varphi))\geq 0$ (see \cite[Fig. 4]{Olver1991}), and therefore \eqref{eq42} implies that the inner sum in \eqref{eq43} is $\O(1)$ for large $|\sigma z|$. Since $k_1+2k_2+\ldots +m k_m=m$, \eqref{eq43} can then be simplified to the form given in \eqref{eq41}. If $\pi < \varphi \le 3\pi - \delta < 3\pi$, then $\Re(c(\varphi))< 0$ (cf. \cite[Fig. 4]{Olver1991}). Hence, in this case the inner sum in \eqref{eq43} is
\[
\O(1)\e^{\frac{j(k_j  - 1)}{2}\left| \sigma z \right|c^2 (\varphi )} ,
\]
giving the desired error term in \eqref{eq41}. 

Suppose now that \eqref{eq41} holds up to $n-1$, where $n$ is a positive integer. The induction step relies upon the identity (see \cite[Eq. (2.8)]{ABOD1997})
\begin{equation}\label{eq44}
\begin{multlined}
 z F^{(m)}\!\left( z;\! \begin{array}{c}
   N ,\,\ldots,\!\!\!\!  \\ 
   \,\sigma ,\,\,\ldots,\!\!\!\!  \\
\end{array}\begin{array}{c}
   N, \!\!\!  \\
   \sigma,  \!\!\!  \\
\end{array}\begin{array}{c}
   N-n \!\!  \\
   \sigma  \!\!  \\
\end{array} \right) = F^{(m)}\!\left( z;\! \begin{array}{c}
   N ,\,\ldots,\!\!\!\!  \\ 
   \,\sigma ,\,\,\ldots,\!\!\!\!  \\
\end{array}\begin{array}{c}
   N, \!\!\!  \\
   \sigma,  \!\!\!  \\
\end{array}\begin{array}{c}
   N-n+1 \!\!  \\
   \sigma  \!\!  \\
\end{array} \right)\\  -\sum_{k=0}^{m-1} F^{(k)}\!\left( z;\! \begin{array}{c}
   N ,\,\ldots,\!\!\!\!  \\ 
   \,\sigma ,\,\,\ldots,\!\!\!\!  \\
\end{array}\begin{array}{c}
   N \!\!  \\
   \sigma  \!\!  \\
\end{array} \right) F^{(m-k)}\!\left( 0;\! \begin{array}{c}
   N+1, \!\!\!  \\
   \sigma,  \!\!\!  \\
\end{array}\begin{array}{c}
   N ,\,\ldots,\!\!\!\!  \\ 
   \,\sigma ,\,\,\ldots,\!\!\!\!  \\
\end{array}\begin{array}{c}
   N-n \!\!  \\
   \sigma  \!\!  \\
\end{array} \right).
\end{multlined}
\end{equation}
We will estimate the sum in \eqref{eq44} under the assumptions that $|\sigma z|$ is large, $N\sim|\sigma z|$, and $m$ and $n$ are both fixed. From Lemma \ref{lemma2}, we can assert that
\begin{multline*}
 \left| F^{(m-k)}\!\left( 0;\! \begin{array}{c}
   N+1, \!\!\!  \\
   \sigma,  \!\!\!  \\
\end{array}\begin{array}{c}
   N ,\,\ldots,\!\!\!\!  \\ 
   \,\sigma ,\,\,\ldots,\!\!\!\!  \\
\end{array}\begin{array}{c}
   N-n \!\!  \\
   \sigma  \!\!  \\
\end{array} \right) \right| \\ \leq c_{m - k} \frac{\sqrt {N - n} \Gamma (N - n)}{\left| \sigma  \right|^{N - n} }\left( \frac{\Gamma (N - 1)}{\left| \sigma  \right|^{N - 1} }\right)^{m - k - 1} \prod\limits_{j = 1}^{m - k - 2}\sqrt {jN - n}
\end{multline*}
for any $0\leq k\leq m-2$. On replacing the gamma functions by Stirling's approximation, we obtain
\begin{equation}\label{eq46}
F^{(m-k)}\!\left( 0;\! \begin{array}{c}
   N+1, \!\!\!  \\
   \sigma,  \!\!\!  \\
\end{array}\begin{array}{c}
   N ,\,\ldots,\!\!\!\!  \\ 
   \,\sigma ,\,\,\ldots,\!\!\!\!  \\
\end{array}\begin{array}{c}
   N-n \!\!  \\
   \sigma  \!\!  \\
\end{array} \right) = \O(1)\left| z \right|^{(m - k)(N-1)- n+1} \e^{ - (m - k)\left| \sigma z\right|} \frac{1}{\sqrt {\left| \sigma z \right|} } .
\end{equation}
By \cite[Eq. (2.2)]{ABOD2009} this estimate is also valid for $k=m-1$. Assume that $\varphi$ is confined to the sector $-\pi < -\pi + \delta \le \varphi \le \pi$ (and so, in particular, $\Re(c(\varphi))\geq 0$). The combination of the base case (with $k$ in place of $m$) and the inequality \eqref{eq42} yields the crude estimate
\begin{equation}\label{eq47}
\frac{\e^{ - k\sigma z} }{(2\pi \im)^k z^{kN - k}} F^{(k)}\!\left( z;\! \begin{array}{c}
   N ,\,\ldots,\!\!\!\!  \\ 
   \,\sigma ,\,\,\ldots,\!\!\!\!  \\
\end{array}\begin{array}{c}
   N \!\!  \\
   \sigma  \!\!  \\
\end{array} \right) = \O(1)\e^{-\frac{k}{2}\left| \sigma z \right|c^2 (\varphi )}.
\end{equation}
Note that \eqref{eq47} remains true when $k=0$. Since $\Re \left( \frac{1}{2}\left| \sigma z\right|c^2 (\varphi ) \right) = \left| \sigma z \right|(1 + \cos \varphi ) = \left| \sigma z \right| + \Re (\sigma z)$, we can infer from \eqref{eq46} and \eqref{eq47} that the summation in \eqref{eq44} is
\begin{multline*}
 \sum\limits_{k = 0}^{m - 1} \O(1)\e^{k\sigma z} z^{kN - k} \e^{ - \frac{k}{2}\left| \sigma z \right|c^2 (\varphi )} \left| z \right|^{(m - k)(N-1) - n + 1} \e^{ - (m - k)\left| \sigma z \right|} \frac{1}{\sqrt {\left| \sigma z \right|}} \\ = z^{mN - m - n + 1} \e^{m\sigma z} \O\bigg( \bigg| \cfrac{\e^{ - \frac{m}{2}\left| \sigma z \right|c^2 (\varphi )} }{\sqrt{|\sigma z|} } \bigg| \bigg).
\end{multline*}
Thus, after re-normalisation, \eqref{eq44} becomes
\begin{multline*}
 \frac{\e^{ - m\sigma z} }{(2\pi \im)^m z^{mN - m - n} }F^{(m)}\!\left( z;\! \begin{array}{c}
   N ,\,\ldots,\!\!\!\!  \\ 
   \,\sigma ,\,\,\ldots,\!\!\!\!  \\
\end{array}\begin{array}{c}
   N, \!\!\!  \\
   \sigma,  \!\!\!  \\
\end{array}\begin{array}{c}
   N-n \!\!  \\
   \sigma  \!\!  \\
\end{array} \right) \\ = \frac{\e^{ - m\sigma z} }{(2\pi \im)^m z^{mN - m - n + 1} }F^{(m)}\!\left( z;\! \begin{array}{c}
   N ,\,\ldots,\!\!\!\!  \\ 
   \,\sigma ,\,\,\ldots,\!\!\!\!  \\
\end{array}\begin{array}{c}
   N, \!\!\!  \\
   \sigma,  \!\!\!  \\
\end{array}\begin{array}{c}
   N-n+1 \!\!  \\
   \sigma  \!\!  \\
\end{array} \right)  + \O\bigg( \bigg| \frac{\e^{ - \frac{m}{2}\left| \sigma z \right|c^2 (\varphi )} }{\sqrt{|\sigma z|} } \bigg| \bigg).
\end{multline*}
The desired result now follows by applying the induction hypothesis on the right-hand side of this equality. If $\pi < \varphi \le 3\pi - \delta < 3\pi$, then $\Re(c(\varphi))< 0$ and the left-hand side of \eqref{eq47} is $\mathcal{O}(1)$. Thus, the sum in \eqref{eq44} may be estimated as
\begin{multline*}
 \sum\limits_{k = 0}^{m - 1} \O(1) \e^{k\sigma z} z^{kN - k} \left| z \right|^{(m - k)(N-1) - n + 1} \e^{ - (m - k)\left| \sigma z \right|} \frac{1}{\sqrt {\left| \sigma z \right|}}\\  = z^{mN - m - n + 1} \e^{m\sigma z} \sum\limits_{k = 0}^{m - 1} \O\bigg( \bigg| \frac{\e^{ - \frac{m-k}{2}\left| \sigma z \right|c^2 (\varphi )} }{\sqrt{|\sigma z|} } \bigg| \bigg) = z^{mN - m - n + 1} \e^{m\sigma z} \O\bigg( \bigg| \frac{\e^{ - \frac{1}{2}\left| \sigma z \right|c^2 (\varphi )} }{\sqrt{|\sigma z|} } \bigg| \bigg).
\end{multline*}
We can now proceed analogously to the preceding case and finish the proof of \eqref{eq41}.

The proof of \eqref{eq45} is completely analogous, except that one uses the asymptotic formula (cf. \cite[Eq. (5.11)]{Olver1991})
\[
\frac{\e^{-2\pi \im N}\e^{ - \sigma z} }{2\pi \im z^{N - 1 - n} }F^{(1)} \!\left( z;\!\begin{array}{c}
   N - n \! \\
   \sigma \! \\
\end{array} \right) = -\frac{1}{2}\erfc\Big( \overline{c(-\varphi )}\sqrt{\tfrac{1}{2}| \sigma z |} \Big) + \mathcal{O}\bigg( \bigg| \frac{\e^{ - \frac{1}{2}| \sigma z |c^2 (-\varphi )} }{\sqrt { \left| \sigma z\right|}} \bigg| \bigg),
\]
valid when $-3\pi<-3\pi+\delta \le \varphi \le \pi-\delta<\pi$, in place of \eqref{eq40}. We omit the details.
\end{proof}

\begin{proof}[Proof of Proposition \ref{prop}] Consider first the case when the singulant pair is $(\sigma, \sigma \e^{-\pi\im})$. If we make a change of integration variable from $t_2$ to $u$ by $u = t_2/(t_1\e^{\pi \im})$ and then from $t_1$ to $s$ by $s = t_1(1 + u)$, \eqref{eq34} is recast as
\[
F^{(2)} \!\left( z;\!\begin{array}{c}
     N,\!\!\!\!   \\ 
     \sigma ,\!\!\!\! \\
\end{array}\begin{array}{c}
    N-n\! \\ 
    \sigma \e^{-\pi\im} \!\\
\end{array} \right) = ( - 1)^n \e^{\pi \im N} \int_0^{ + \infty } \frac{u^{N - n - 1} }{(1 + u)^{2N - n - 1}}F^{(1)}\! \left( (1 + u)z;\!\begin{array}{c}
   2N - n - 1 \! \\
   \sigma  \! \\
\end{array} \right)\d u.
\]
Using the inequality \eqref{eq38} and the known integral representation of the beta function \cite[\href{https://dlmf.nist.gov/5.12.E3}{Eq. 5.12.3}]{NIST:DLMF}, we deduce
\begin{equation}\label{eq50}
\left| F^{(2)} \!\left( z;\!\begin{array}{c}
     N,\!\!\!\!   \\ 
     \sigma ,\!\!\!\! \\
\end{array}\begin{array}{c}
    N-n\! \\ 
    \sigma \e^{-\pi\im} \!\\
\end{array} \right) \right| \le c_1 \frac{1}{\left| z \right|}\frac{1}{\left| \sigma  \right|^{2N - n - 1} }\frac{\Gamma (N - n)\Gamma (N)}{\sqrt {2N - n - 1} }
\end{equation}
provided $|\varphi|\leq \pi$. If $|\sigma z|$ is large, $N\sim|\sigma z|$, and $n$ is fixed, then, with the aid of Stirling's formula, we find from \eqref{eq50} that
\begin{equation}\label{eq51}
F^{(2)} \!\left( z;\!\begin{array}{c}
     N,\!\!\!\!   \\ 
     \sigma ,\!\!\!\! \\
\end{array}\begin{array}{c}
    N-n\! \\ 
    \sigma \e^{-\pi\im} \!\\
\end{array} \right)= \O(1)\left| z \right|^{2N -2- n} \e^{ - 2\left| \sigma z\right|} \frac{1}{\sqrt {\left| \sigma z\right|} } = z^{2N -2- n } \e^{2\sigma z} \O\bigg( \bigg| \frac{\e^{ - \left| \sigma z \right|c^2 (\varphi )} }{\sqrt{|\sigma z|} } \bigg| \bigg)
\end{equation}
uniformly in the sector $|\varphi| \leq \pi$.

Suppose now that $\pi  < \varphi \le 2\pi - \delta<2\pi$. From \eqref{eq49}, one infers
\[
F^{(2)} \!\left( z;\!\begin{array}{c}
     N,\!\!\!\!   \\ 
     \sigma ,\!\!\!\! \\
\end{array}\begin{array}{c}
    N-n\! \\ 
    \sigma \e^{-\pi\im} \!\\
\end{array} \right)= F^{(2)} \!\left( z\e^{-2\pi \im};\!\begin{array}{c}
     N,\!\!\!\!   \\ 
     \sigma ,\!\!\!\! \\
\end{array}\begin{array}{c}
    N-n\! \\ 
    \sigma \e^{-\pi\im} \!\\
\end{array} \right) + 2\pi \im\e^{\sigma z} z^{N - 1} F^{(1)}\! \left( z;\!\begin{array}{c}
   N - n \! \\
   \sigma \e^{ - \pi \im} \! \\
\end{array} \right).
\]
If we again take $N\sim |\sigma z| \gg 1$, and keep $n$ fixed, then we can estimate the first term on the right-hand side by the quantity on the right-hand side of \eqref{eq51}. Regarding the second term, since $0< \arg(\sigma \e^{ - \pi \im}z) \le \pi - \delta<\pi$, Olver's formula \eqref{eq52} implies that
\[
2\pi \im\e^{\sigma z} z^{N - 1} F^{(1)}\! \left( z;\!\begin{array}{c}
   N - n \! \\
   \sigma \e^{ - \pi \im} \! \\
\end{array} \right) = z^{2N-2-n} \e^{2\sigma z} \O\bigg( \bigg| \frac{\e^{ - \frac{1}{2}\left| \sigma z \right|c^2 (\varphi )} }{\sqrt{|\sigma z|} } \bigg| \bigg).
\]
The conjugate sector $-2\pi < -2\pi + \delta \le \varphi < -\pi$ may be treated in a similar manner.

If the singulant pair is $(\sigma, \sigma \e^{\pi\im})$, the estimate \eqref{eq55} can simply be deduced from the previous case and the equality
\[
F^{(2)} \!\left( z;\!\begin{array}{c}
     N,\!\!\!\!   \\ 
     \sigma ,\!\!\!\! \\
\end{array}\begin{array}{c}
    N-n\! \\ 
    \sigma \e^{\pi\im} \!\\
\end{array} \right) = \e^{-2\pi \im N} F^{(2)} \!\left( z;\!\begin{array}{c}
     N,\!\!\!\!   \\ 
     \sigma ,\!\!\!\! \\
\end{array}\begin{array}{c}
    N-n\! \\ 
    \sigma \e^{-\pi\im} \!\\
\end{array} \right),
\]
which follows readily from the definition \eqref{eq34}.
\end{proof}

\section{Conclusions}\label{section6}

The Stokes phenomenon concerns the sudden change across certain rays in the complex plane, known as Stokes lines, exhibited by the coefficients multiplying exponentially small terms in compound asymptotic expansions. Dingle introduced a set of rules for locating Stokes lines and continuing asymptotic expansions across them. Included among these rules is the ``final main rule" stating that half the discontinuity in form occurs on reaching the Stokes line, and half on leaving it the other side. Berry showed that, if an asymptotic expansion is terminated at or near its numerically least term, the transition between two different asymptotic forms across a Stokes line is effected smoothly and not discontinuously. Furthermore, for a wide class of functions, the coefficient multiplying a subdominant exponential contribution (a Stokes multiplier) possesses a universal structure expressed approximately in terms of an error function whose argument is an appropriate variable describing the transition across a Stokes line.

In this paper, we revisited the well-known asymptotic expansions of the gamma function and its reciprocal. These expansions do not share the simple properties above. In the neighbourhood of a Stokes line, not one but infinitely many exponentially small contributions appear, each associated with its own Stokes multiplier. Moreover, these multipliers may no longer obey Dingle's rule: their values can differ from $\frac{1}{2}$ on a Stokes line and can be non-zero only on the line itself. This unconventional behaviour of the multipliers is a manifestation of an infinite number of higher-order Stokes phenomena. We demonstrated that these phenomena are rapid but smooth transitions in the remainder terms of successive hyperasymptotic re-expansions. The approximate functional form of the Stokes multipliers is, however, no longer described by a single error function but rather by a multivariate polynomial in such functions. The basis of our proof was an identity between hyperterminants, which was originally conjectured by Howls.

The treatment presented here is limited to the situation where all the singulants are equal to each other. In the language of Borel summation, this means that the singularities of the Borel transform are collinear and are equally spaced. This is a common phenomenon, e.g., for transseries solutions to non-linear difference and differential equations (cf. \cite{Boele2001,Braaksma2004,ABOD2005}). The work in this paper can certainly be extended to cover such problems. A more interesting extension is to allow the singulants to have different magnitude but remain collinear. We believe that the smooth interpretation of the higher-order Stokes phenomenon remains possible in such circumstances, although the approximate functional form may involve transcendental functions that are more complicated than an error function.

\section*{Acknowledgement} The author thanks the referees for helpful comments and suggestions for improving the presentation. The author wishes to thank A.~B.~Olde Daalhuis for useful discussions. The author's research was supported by a Premium Postdoctoral Fellowship of the Hungarian Academy of Sciences and by the JSPS Postdoctoral Research Fellowship No. P21020.

\medskip


\begin{thebibliography}{10}

\bibitem{Bell1934}
E.~T.~Bell, Exponential polynomials, \emph{Ann. Math.} \textbf{35} (1934), no. 2, pp. 258--277.

\bibitem{BHNOD2018}
T.~B.~Bennett, C.~J.~Howls, G.~Nemes, A.~B.~Olde Daalhuis, Globally exact asymptotics for integrals with arbitrary order saddles, \emph{SIAM J. Math. Anal.} \textbf{50} (2018), no. 2, pp. 2144--2177.

\bibitem{Berry1989}
M.~V.~Berry, Uniform asymptotic smoothing of Stokes's discontinuities, \emph{Proc. R. Soc. Lond. A} \textbf{422} (1989), no. 1862, pp. 7--21.

\bibitem{Berry1991}
M.~V.~Berry, Infinitely many Stokes smoothings in the gamma function, \emph{Proc. R. Soc. Lond. A} \textbf{434} (1991), no. 1891, pp. 465--472.

\bibitem{BerryHowls1990}
M.~V.~Berry, C.~J.~Howls, Hyperasymptotics, \emph{Proc. R. Soc. Lond. A} \textbf{430} (1990), no. 1880, pp. 653--668.

\bibitem{BerryHowls1991}
M.~V.~Berry, C.~J.~Howls, Hyperasymptotics for integrals with saddles, \emph{Proc. R. Soc. Lond. A} \textbf{434} (1991), no. 1892, pp. 657--675.

\bibitem{Boele2001}
L.~J.~Boele, Transseries for a class of nonlinear difference equations, \emph{J. Differ. Equ. Appl.} \textbf{7} (2001), no. 5, pp. 717--750.

\bibitem{Boyd1994}
W.~G.~C.~Boyd, Gamma function asymptotics by an extension of the method of steepest descents, \emph{Proc. R. Soc. Lond. A} \textbf{447} (1994), no. 1931, pp. 609--630.

\bibitem{Braaksma2004}
B.~Braaksma, R.~Kuik, Resurgence relations for classes of differential and difference equations, \emph{Ann. Fac. Sci. Toulouse Math.} \textbf{13} (2004), no. 4, pp. 479--492.

\bibitem{Chapman1996}
S.~J.~Chapman, On the non-universality of the error function in the smoothing of Stokes discontinuities, \emph{Proc. R. Soc. Lond. A} \textbf{452} (1996), no. 1953, pp. 2225--2230.

\bibitem{Darboux1878}
G.~M.~Darboux, M\'emoire sur l'approximation des fonctions de tr\`es-grandes nombres, et sur une classe \'etendue de d\'eveloppements en s\'erie, \emph{J. Math. Pures
Appl.} \textbf{4} (1878), no. 3, 6--56, pp. 377--416.

\bibitem{Dingle1973}
R.~B.~Dingle, \emph{Asymptotic Expansions: Their Derivation and Interpretation}, Academic Press, London--New York, 1973.

\bibitem{Howls2004}
C.~J.~Howls, P.~J.~Langman, A.~B.~Olde Daalhuis, On the higher-order Stokes phenomenon, \emph{Proc. R. Soc. Lond. A} \textbf{460} (2004), no. 2048, pp. 2285--2303.

\bibitem{Nemes2015}
G.~Nemes, Error bounds and exponential improvements for the asymptotic expansions of the gamma function and its reciprocal, \emph{Proc. R. Soc. Edinb. A: Math.} \textbf{145} (2015), no. 3, pp. 571--596.

\bibitem{ABOD1997}
A.~B.~Olde Daalhuis, Hyperterminants II, \emph{J. Comput. Appl. Math.} \textbf{89} (1998), no. 1, pp. 87--95.

\bibitem{ABOD1998}
A.~B.~Olde Daalhuis, Hyperasymptotic solutions of higher order linear differential equations with a singularity of rank one, \emph{Proc. R. Soc. Lond. A} \textbf{454} (1998), no. 1968, pp. 1--29.

\bibitem{ABOD2005}
A.~B.~Olde Daalhuis, Hyperasymptotics for nonlinear ODEs II. The first Painlev\'e equation and a second-order Riccati equation, \emph{Proc. R. Soc. Lond. A} \textbf{461} (2005), no. 2062, pp. 3005--3021.

\bibitem{ABOD2009}
A.~B.~Olde Daalhuis, Hyperasymptotics and hyperterminants: exceptional cases, \emph{J. Comput. Appl. Math.} \textbf{233} (2009), no. 2, pp. 555--563.

\bibitem{ABOD1994}
A.~B.~Olde Daalhuis, F.~W.~J.~Olver, Exponentially improved asymptotic solutions of ordinary differential equations. II. Irregular singularities of rank one, \emph{Proc. R. Soc. Lond. A} \textbf{445} (1994), no. 1923, pp. 39--56.

\bibitem{Olver1991}
F.~W.~J.~Olver, Uniform, exponentially improved, asymptotic expansions for the generalized exponential integral, \emph{SIAM J. Math. Anal.} \textbf{22} (1991), no. 5, pp. 1460--1474.

\bibitem{NIST:DLMF}
\emph{NIST Digital Library of Mathematical Functions}. \url{https://dlmf.nist.gov/}, Release 1.1.7 of 2022-10-15. F.~W.~J.~Olver, A.~B.~Olde Daalhuis, D.~W.~Lozier, B.~I.~Schneider, R.~F.~Boisvert, C.~W.~Clark, B.~R.~Miller, B.~V.~Saunders, H.~S.~Cohl, and M.~A.~McClain, eds.

\bibitem{Paris1993}
R.~B.~Paris, Application of the refined asymptotics of the gamma function to the Riemann zeta function, Technical Report MACS 93:11, Department of Mathematical and Computer Science, Dundee Institute of Technology, 1993.

\bibitem{PW1992}
R.~B.~Paris, A.~D.~Wood, Exponentially-improved asymptotics for the gamma function, \emph{J. Comput. Appl. Math.} \textbf{41} (1992), no. 1--2, pp. 135--143.

\bibitem{Paris2001}
R.~B.~Paris, D.~Kaminski, \emph{Asymptotics and Mellin--Barnes Integrals}, Encyclopedia of Mathematics and its Applications, Cambridge University Press, Cambridge, 2001.

\bibitem{Poincare1886}
H.~Poincar\'e, Sur les int\'egrales irr\'eguli\`eres des \'equations lin\'eaires, \emph{Acta Math.} \textbf{8} (1886), pp. 295--344.

\bibitem{Riordan1968}
J.~Riordan, \emph{Combinatorial Identities}, John Wiley \& Sons, Inc., New York--London--Sydney, 1968.

\bibitem{Stokes1864}
G.~G.~Stokes, On the discontinuity of arbitrary constants which appear in divergent developments, \emph{Trans. Camb. Philos. Soc.} \textbf{10} (1864), pp. 106--128.

\bibitem{Temme1996}
N.~M.~Temme, \emph{Special Functions: An Introduction to the Classical Functions of Mathematical Physics}, John Wiley \& Sons, Inc., New York, 1996.

\bibitem{WW1927}
E.~T.~Whittaker, G.~N.~Watson, \emph{A Course of Modern Analysis}, Cambridge University Press, Cambridge, 1927.

\end{thebibliography}
\end{document}